\documentclass{amsart}
\usepackage[dvipdfmx]{graphicx}
\usepackage[varg]{txfonts}
\usepackage{mathtools,amssymb,amsthm,enumerate,comment,url}
\usepackage[arrow,curve,matrix]{xy}
\usepackage{fullpage}
\theoremstyle{plain}

\newtheorem{thm}{Theorem}[section]
\newtheorem{lem}[thm]{Lemma}
\newtheorem{prop}[thm]{Proposition}

\theoremstyle{definition}

\newtheorem{defn}[thm]{Definition}
\newtheorem{rem}[thm]{Remark}
\newtheorem{ex}[thm]{Example}

\newcommand{\R}{\mathbb{R}}

\newcommand{\X}{\mathcal{X}}
\newcommand{\Y}{\mathcal{Y}}

\newcommand{\abs}[1]{\left\lvert {#1} \right\rvert}

\newcommand{\id}{\mathrm{id}}

\newcommand{\link}{\mathrm{link}}

\renewcommand{\O}{\mathrm{O}}
\newcommand{\pair}[2]{\langle #1 ,\, #2 \rangle}
\newcommand{\pr}{\mathrm{pr}}

\newcommand{\sgn}{\mathrm{sgn}}
\newcommand{\supp}{\mathrm{supp}}

\numberwithin{equation}{section}
\numberwithin{figure}{section}

\allowdisplaybreaks[1]

\title{Polyak-Viro type formula for the Milnor triple linking number of link diagrams with multiple-crossings}
\author{Yusaku Okuhara and Keiichi Sakai}
\address{Faculty of Mathematics, Shinshu University, 3-1-1 Asahi, Matsumoto, Nagano 390-8621, Japan}
\email{23ss103j@gmail.com}
\email{sakaik@shinshu-u.ac.jp}
\date{\today}

\begin{document}
\begin{abstract}
We obtain Polyak-Viro type formula for the Milnor triple linking number that can be applied to diagrams with triple or more multiple-crossings.
The proof is based on the idea of Brooks and Komendarczyk \cite{BrooksKomendarczyk24}, but is different from theirs in that we explicitly compute the value of configuration space integral associated to the ``Y-graph,'' and is applicable to the original Brooks-Komendarczyk formula for the Casson knot invariant.
The results for the Milnor triple linking number were firstly obtained in the first author's master thesis, in which the proof follows the same way as that in \cite{BrooksKomendarczyk24}.
\end{abstract}
\maketitle

\section{Introduction}
Configuration space integral \cite{AltschulerFreidel97,BottTaubes94,CCL02,Kohno94,Volic05} provides us a systematic way to express all the Vassiliev invariants for (long) knots and links.
It also enables us to compute Vassiliev invariants in combinatorial ways;
the Polyak-Viro type formulas \cite{PolyakViro94,PolyakViro01} can be deduced from configuration space integrals.
Moreover it motivates us to think of invariants for high dimensional embeddings as generalizations of Vassiliev invariants, see for example \cite{K14_2,KWatanabe08,Watanabe07}.

Robyn Brooks and Rafal Komendarczyk \cite{BrooksKomendarczyk24} got an idea to replace the standard volume form of $S^2$ with a 2-form whose support is localized to a neighborhood of a point (say the north pole).
This idea can work well only for the Vassiliev invariants of order two, but extremely simplifies the computation of configuration space integral since, with such a 2-form, the integrals ``localize'' to neighborhoods of crossings of knot diagrams and get to be computable in a combinatorial way.
Indeed Brooks and Komendarczyk were able to deduce Polyak-Viro type formula for the \emph{Casson knot invariant} that is applicable to diagrams with triple or more multiple crossings.
They firstly proved the original Polyak-Viro formula for the Casson invariant \cite{PolyakViro01} using configuration space integral, and then addressed the case of diagrams with multiple crossings by resolving the multiple crossings into ordinary double crossings to apply the original Polyak-Viro formula.
Such resolutions allow us to avoid the computation of configuration space integral associated to the ``Y-graph'' (see Figure~\ref{fig:Y_link} below).

In this paper we give a similar combinatorial formula for the \emph{Milnor triple linking number} of 3-component long links (Theorem~\ref{thm:main}).
The proof is different from that in \cite{BrooksKomendarczyk24} in that we explicitly compute the Y-graph integral.
Our computation clarifies the nature of the Y-graph as a ``correction term'' to some extent.
Our method can also be applied to reprove the formula in \cite{BrooksKomendarczyk24}, and might be useful for further study of Vassiliev invariants of order two.

The paper is organized as follows.
In \S\ref{s:Prerimilaries} we introduce arrow / Gauss diagrams that are needed to state our main results (Theorem~\ref{thm:main}).
\S\ref{s:CSI} describes three invariants using configuration space integrals; the linking number for 2-component links, the Casson knot invariant, and the Milnor triple linking number.
We review the integral expression of the linking numbers of 2-component links for two reasons.
One is that the linking number can be seen as the simplest example of configuration space integrals,
and the other is that we obtain a ``localizing technique'' of integrations through an explicit computation of the linking number.
In \S\ref{s:CSI_multiple_crossings} we give the proof of the main results by computing configuration space integrals in the case where diagrams have multiple crossings.

\section*{Acknowledgments}
The authors would like to express their great appreciation to Robyn Brooks and Rafal Komendarczyk for helpful comments and suggestions.
The secound author was partially supported by JSPS KAKENHI Grant Numbers 20K03608 and 23K20795.

\section{Preliminaries and the main results}\label{s:Prerimilaries}
\begin{defn}
A \emph{long knot} in $\R^3$ is an embedding $f\colon\R^1\hookrightarrow\R^3$ that satisfies $f(x)=(x,0,0)$ if $\abs{x}\ge 1$.
A \emph{3-component long link} in $\R^3$ is an embedding $f=f_1\sqcup f_2\sqcup f_3\colon\R^1\sqcup\R^1\sqcup\R^1\hookrightarrow\R^3$ that satisfies $f_k(x)=(x,(2-k)\abs{x},0)$ if $\abs{x}\ge 1$.
See Figure~\ref{fig:long_embeddings}.
\end{defn}
\begin{figure}
\centering
\input{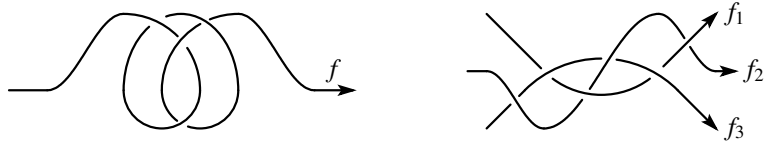}
\caption{Long knot and 3-component long link}
\label{fig:long_embeddings}
\end{figure}

Let $N=(0,0,1)\in\R^3$, and let $N^{\perp}$ be the $2$-dimensional subspace of $\R^3$ perpendicular to $N$.
We denote the projection from $\R^3$ onto $N^{\perp}$ (resp.~$\R N$) by $\pr_{12}$ (resp.~$\pr_3$).

In this paper, unless otherwise stated, a \emph{diagram} of a (long) knot or link $f$ in $\R^3$ is the projection of $f$ onto $N^{\perp}$ with over / under information, that may have \emph{transverse multiple crossings}.
Namely,  $\pr_{12}(f)$ may go through a point twice or more times, but all the tangent vectors at the point should generate mutually different lines.

\begin{rem}
We think of a multiple crossing as consisting of several fewer multiple crossings; for example, a triple crossing consists of three double crossings, and a quadruple crossing contains four triple crossings and six double crossings.
\end{rem}

Let $f$ be a long knot with a diagram $D$.
The \emph{arrow diagram $A_D$} for $D$ is the set of ordered pairs $(t_i,u_i)$ ($i=1,2,\dots$) of real numbers, where each $(t_i,u_i)$ corresponds to a double crossing $\pr_{12}(f(t_i))=\pr_{12}(f(u_i))$ of $D$ with $\pr_3(f(t_i))<\pr_3(f(u_i))$.
We depict $A_D$ as a graph that has vertices $t_i$ and $u_i$ ($i=1,2,\dots$) on $\R^1$ with oriented chords emanating from $t_i$ to $u_i$ (see Figures~\ref{fig:example_arrow_Gauss_double}, \ref{fig:example_arrow_Gauss_triple}).
The \emph{Gauss diagram} $G_D$ is $A_D$ with the signs of the double crossings of $D$ attached to corresponding chords.
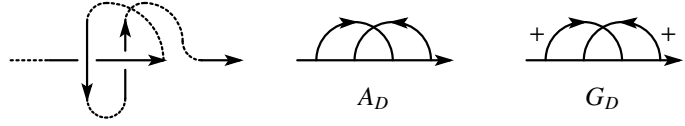
\begin{figure}
\centering
{\unitlength 0.1in%
\begin{picture}(35.0000,6.0000)(2.0000,-20.0000)%
%
\special{pn 13}%
\special{pa 600 1500}%
\special{pa 600 1900}%
\special{fp}%
\special{sh 1}%
\special{pa 600 1900}%
\special{pa 620 1833}%
\special{pa 600 1847}%
\special{pa 580 1833}%
\special{pa 600 1900}%
\special{fp}%
%
\special{pn 13}%
\special{pa 400 1700}%
\special{pa 550 1700}%
\special{fp}%
%
\special{pn 13}%
\special{pa 650 1700}%
\special{pa 1000 1700}%
\special{fp}%
\special{sh 1}%
\special{pa 1000 1700}%
\special{pa 933 1680}%
\special{pa 947 1700}%
\special{pa 933 1720}%
\special{pa 1000 1700}%
\special{fp}%
%
\special{pn 13}%
\special{pa 200 1700}%
\special{pa 400 1700}%
\special{dt 0.030}%
%
\special{pn 13}%
\special{pn 13}%
\special{pa 800 1900}%
\special{pa 799 1912}%
\special{fp}%
\special{pa 796 1927}%
\special{pa 793 1937}%
\special{fp}%
\special{pa 787 1950}%
\special{pa 781 1958}%
\special{fp}%
\special{pa 773 1968}%
\special{pa 767 1975}%
\special{fp}%
\special{pa 755 1983}%
\special{pa 746 1989}%
\special{fp}%
\special{pa 732 1995}%
\special{pa 722 1998}%
\special{fp}%
\special{pa 707 2000}%
\special{pa 694 2000}%
\special{fp}%
\special{pa 679 1998}%
\special{pa 668 1995}%
\special{fp}%
\special{pa 654 1989}%
\special{pa 645 1983}%
\special{fp}%
\special{pa 634 1975}%
\special{pa 627 1969}%
\special{fp}%
\special{pa 618 1958}%
\special{pa 613 1949}%
\special{fp}%
\special{pa 607 1935}%
\special{pa 604 1926}%
\special{fp}%
\special{pa 601 1912}%
\special{pa 600 1900}%
\special{fp}%
%
\special{pn 13}%
\special{pn 13}%
\special{pa 800 1500}%
\special{pa 801 1486}%
\special{fp}%
\special{pa 804 1470}%
\special{pa 809 1459}%
\special{fp}%
\special{pa 817 1445}%
\special{pa 824 1435}%
\special{fp}%
\special{pa 835 1424}%
\special{pa 844 1417}%
\special{fp}%
\special{pa 858 1409}%
\special{pa 871 1404}%
\special{fp}%
\special{pa 887 1401}%
\special{pa 900 1400}%
\special{fp}%
%
\special{pn 13}%
\special{pa 1200 1700}%
\special{pa 1400 1700}%
\special{fp}%
\special{sh 1}%
\special{pa 1400 1700}%
\special{pa 1333 1680}%
\special{pa 1347 1700}%
\special{pa 1333 1720}%
\special{pa 1400 1700}%
\special{fp}%
%
\special{pn 13}%
\special{pa 1700 1700}%
\special{pa 2500 1700}%
\special{fp}%
\special{sh 1}%
\special{pa 2500 1700}%
\special{pa 2433 1680}%
\special{pa 2447 1700}%
\special{pa 2433 1720}%
\special{pa 2500 1700}%
\special{fp}%
%
\special{pn 13}%
\special{pa 800 1900}%
\special{pa 800 1750}%
\special{fp}%
%
\special{pn 13}%
\special{pa 800 1650}%
\special{pa 800 1500}%
\special{fp}%
\special{sh 1}%
\special{pa 800 1500}%
\special{pa 780 1567}%
\special{pa 800 1553}%
\special{pa 820 1567}%
\special{pa 800 1500}%
\special{fp}%
%
\special{pn 13}%
\special{pn 13}%
\special{pa 700 1400}%
\special{pa 713 1400}%
\special{fp}%
\special{pa 729 1401}%
\special{pa 741 1403}%
\special{fp}%
\special{pa 757 1406}%
\special{pa 769 1408}%
\special{fp}%
\special{pa 785 1412}%
\special{pa 797 1416}%
\special{fp}%
\special{pa 812 1422}%
\special{pa 823 1427}%
\special{fp}%
\special{pa 838 1433}%
\special{pa 848 1439}%
\special{fp}%
\special{pa 862 1448}%
\special{pa 873 1455}%
\special{fp}%
\special{pa 886 1465}%
\special{pa 895 1472}%
\special{fp}%
\special{pa 907 1483}%
\special{pa 916 1492}%
\special{fp}%
\special{pa 927 1504}%
\special{pa 935 1514}%
\special{fp}%
\special{pa 945 1527}%
\special{pa 952 1538}%
\special{fp}%
\special{pa 961 1551}%
\special{pa 967 1562}%
\special{fp}%
\special{pa 973 1577}%
\special{pa 978 1588}%
\special{fp}%
\special{pa 984 1603}%
\special{pa 988 1615}%
\special{fp}%
\special{pa 992 1631}%
\special{pa 994 1643}%
\special{fp}%
\special{pa 997 1659}%
\special{pa 998 1671}%
\special{fp}%
\special{pa 1000 1687}%
\special{pa 1000 1700}%
\special{fp}%
%
\special{pn 13}%
\special{pn 13}%
\special{pa 600 1500}%
\special{pa 601 1486}%
\special{fp}%
\special{pa 604 1470}%
\special{pa 609 1459}%
\special{fp}%
\special{pa 617 1445}%
\special{pa 624 1435}%
\special{fp}%
\special{pa 635 1424}%
\special{pa 644 1417}%
\special{fp}%
\special{pa 658 1409}%
\special{pa 671 1404}%
\special{fp}%
\special{pa 687 1401}%
\special{pa 700 1400}%
\special{fp}%
%
\special{pn 13}%
\special{pn 13}%
\special{pa 900 1400}%
\special{pa 914 1400}%
\special{fp}%
\special{pa 931 1402}%
\special{pa 943 1405}%
\special{fp}%
\special{pa 959 1409}%
\special{pa 972 1413}%
\special{fp}%
\special{pa 987 1420}%
\special{pa 999 1426}%
\special{fp}%
\special{pa 1013 1435}%
\special{pa 1024 1443}%
\special{fp}%
\special{pa 1036 1454}%
\special{pa 1046 1463}%
\special{fp}%
\special{pa 1057 1476}%
\special{pa 1064 1486}%
\special{fp}%
\special{pa 1073 1501}%
\special{pa 1079 1512}%
\special{fp}%
\special{pa 1086 1528}%
\special{pa 1091 1540}%
\special{fp}%
\special{pa 1095 1556}%
\special{pa 1098 1569}%
\special{fp}%
\special{pa 1100 1587}%
\special{pa 1100 1600}%
\special{fp}%
%
\special{pn 13}%
\special{pn 13}%
\special{pa 1200 1700}%
\special{pa 1186 1699}%
\special{fp}%
\special{pa 1170 1696}%
\special{pa 1160 1692}%
\special{fp}%
\special{pa 1145 1684}%
\special{pa 1136 1677}%
\special{fp}%
\special{pa 1126 1667}%
\special{pa 1118 1658}%
\special{fp}%
\special{pa 1110 1643}%
\special{pa 1105 1630}%
\special{fp}%
\special{pa 1101 1613}%
\special{pa 1100 1600}%
\special{fp}%
%
\special{pn 13}%
\special{ar 2000 1700 200 200 3.1415927 4.7123890}%
%
\special{pn 13}%
\special{pa 1985 1501}%
\special{pa 2000 1500}%
\special{fp}%
\special{sh 1}%
\special{pa 2000 1500}%
\special{pa 1932 1484}%
\special{pa 1947 1504}%
\special{pa 1935 1524}%
\special{pa 2000 1500}%
\special{fp}%
%
\special{pn 13}%
\special{ar 2200 1700 200 200 4.7123890 6.2831853}%
%
\special{pn 13}%
\special{pa 2215 1501}%
\special{pa 2200 1500}%
\special{fp}%
\special{sh 1}%
\special{pa 2200 1500}%
\special{pa 2265 1524}%
\special{pa 2253 1504}%
\special{pa 2268 1484}%
\special{pa 2200 1500}%
\special{fp}%
%
\special{pn 13}%
\special{ar 2000 1700 200 200 4.7123890 6.2831853}%
%
\special{pn 13}%
\special{ar 2200 1700 200 200 3.1415927 4.7123890}%
\put(21.0000,-19.0000){\makebox(0,0){$A_D$}}%
%
\special{pn 13}%
\special{pa 2900 1700}%
\special{pa 3700 1700}%
\special{fp}%
\special{sh 1}%
\special{pa 3700 1700}%
\special{pa 3633 1680}%
\special{pa 3647 1700}%
\special{pa 3633 1720}%
\special{pa 3700 1700}%
\special{fp}%
%
\special{pn 13}%
\special{ar 3200 1700 200 200 3.1415927 4.7123890}%
%
\special{pn 13}%
\special{pa 3185 1501}%
\special{pa 3200 1500}%
\special{fp}%
\special{sh 1}%
\special{pa 3200 1500}%
\special{pa 3132 1484}%
\special{pa 3147 1504}%
\special{pa 3135 1524}%
\special{pa 3200 1500}%
\special{fp}%
%
\special{pn 13}%
\special{ar 3200 1700 200 200 4.7123890 6.2831853}%
%
\special{pn 13}%
\special{ar 3400 1700 200 200 3.1415927 4.7123890}%
%
\special{pn 13}%
\special{ar 3400 1700 200 200 4.7123890 6.2831853}%
%
\special{pn 13}%
\special{pa 3415 1501}%
\special{pa 3400 1500}%
\special{fp}%
\special{sh 1}%
\special{pa 3400 1500}%
\special{pa 3465 1524}%
\special{pa 3453 1504}%
\special{pa 3468 1484}%
\special{pa 3400 1500}%
\special{fp}%
\put(33.0000,-19.0000){\makebox(0,0){$G_D$}}%
\put(30.0000,-16.0000){\makebox(0,0)[rb]{$+$}}%
\put(36.0000,-16.0000){\makebox(0,0)[lb]{$+$}}%
\end{picture}}%
\caption{Examples of arrow / Gauss diagrams for double crossings}
\label{fig:example_arrow_Gauss_double}
\end{figure}
\begin{figure}
\centering
\input{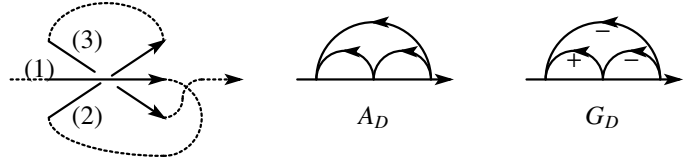}
\caption{Examples of arrow / Gauss diagrams for triple crossings; the arc (2) is front of (3)}
\label{fig:example_arrow_Gauss_triple}
\end{figure}

The \emph{3-strand arrow diagram} and the \emph{3-strand Gauss diagram} of a diagram $D$ of a 3-component long link are defined in the same way.
It is a union of three copies of $\R^1$ with vertices on these $\R^1$'s connected by oriented chords corresponding to double crossings (and the sings of these crossings are attached to these chords, see Figure~\ref{fig:example_arrow_Gauss_3-strand}).
\begin{figure}
\centering
{\unitlength 0.1in%
\begin{picture}(33.0000,5.5000)(4.0000,-9.2700)%
%
\special{pn 13}%
\special{pa 750 700}%
\special{pa 1000 700}%
\special{fp}%
\special{sh 1}%
\special{pa 1000 700}%
\special{pa 933 680}%
\special{pa 947 700}%
\special{pa 933 720}%
\special{pa 1000 700}%
\special{fp}%
%
\special{pn 13}%
\special{pa 400 900}%
\special{pa 1000 500}%
\special{fp}%
\special{sh 1}%
\special{pa 1000 500}%
\special{pa 933 520}%
\special{pa 956 530}%
\special{pa 956 554}%
\special{pa 1000 500}%
\special{fp}%
%
\special{pn 13}%
\special{pa 400 500}%
\special{pa 640 660}%
\special{fp}%
%
\special{pn 13}%
\special{pa 760 740}%
\special{pa 1000 900}%
\special{fp}%
\special{sh 1}%
\special{pa 1000 900}%
\special{pa 956 846}%
\special{pa 956 870}%
\special{pa 933 880}%
\special{pa 1000 900}%
\special{fp}%
%
\special{pn 13}%
\special{pa 1700 700}%
\special{pa 1900 700}%
\special{fp}%
\special{sh 1}%
\special{pa 1900 700}%
\special{pa 1833 680}%
\special{pa 1847 700}%
\special{pa 1833 720}%
\special{pa 1900 700}%
\special{fp}%
%
\special{pn 13}%
\special{ar 1950 700 150 150 4.7123890 6.2831853}%
%
\special{pn 13}%
\special{ar 2250 700 150 150 4.7123890 6.2831853}%
%
\special{pn 13}%
\special{pa 2265 551}%
\special{pa 2250 550}%
\special{fp}%
\special{sh 1}%
\special{pa 2250 550}%
\special{pa 2315 574}%
\special{pa 2303 554}%
\special{pa 2318 534}%
\special{pa 2250 550}%
\special{fp}%
%
\special{pn 13}%
\special{ar 2100 700 300 300 4.7123890 6.2831853}%
%
\special{pn 13}%
\special{pa 2115 400}%
\special{pa 2100 400}%
\special{fp}%
\special{sh 1}%
\special{pa 2100 400}%
\special{pa 2167 420}%
\special{pa 2153 400}%
\special{pa 2167 380}%
\special{pa 2100 400}%
\special{fp}%
\put(11.0000,-7.0000){\makebox(0,0){$f_1$}}%
\put(11.0000,-5.0000){\makebox(0,0){$f_2$}}%
\put(11.0000,-9.0000){\makebox(0,0){$f_3$}}%
%
\special{pn 13}%
\special{ar 1950 700 150 150 3.1415927 4.7123890}%
%
\special{pn 13}%
\special{pa 1935 551}%
\special{pa 1950 550}%
\special{fp}%
\special{sh 1}%
\special{pa 1950 550}%
\special{pa 1882 534}%
\special{pa 1897 554}%
\special{pa 1885 574}%
\special{pa 1950 550}%
\special{fp}%
\put(31.5000,-6.0000){\makebox(0,0){$-$}}%
%
\special{pn 13}%
\special{pa 2900 700}%
\special{pa 3100 700}%
\special{fp}%
\special{sh 1}%
\special{pa 3100 700}%
\special{pa 3033 680}%
\special{pa 3047 700}%
\special{pa 3033 720}%
\special{pa 3100 700}%
\special{fp}%
%
\special{pn 13}%
\special{ar 2250 700 150 150 3.1415927 4.7123890}%
%
\special{pn 13}%
\special{ar 2100 700 300 300 3.1415927 4.7123890}%
%
\special{pn 13}%
\special{ar 3300 700 300 300 4.7123890 6.2831853}%
%
\special{pn 13}%
\special{pa 3315 400}%
\special{pa 3300 400}%
\special{fp}%
\special{sh 1}%
\special{pa 3300 400}%
\special{pa 3367 420}%
\special{pa 3353 400}%
\special{pa 3367 380}%
\special{pa 3300 400}%
\special{fp}%
%
\special{pn 13}%
\special{ar 3300 700 300 300 3.1415927 4.7123890}%
%
\special{pn 13}%
\special{ar 3150 700 150 150 4.7123890 6.2831853}%
%
\special{pn 13}%
\special{ar 3150 700 150 150 3.1415927 4.7123890}%
%
\special{pn 13}%
\special{pa 3135 551}%
\special{pa 3150 550}%
\special{fp}%
\special{sh 1}%
\special{pa 3150 550}%
\special{pa 3082 534}%
\special{pa 3097 554}%
\special{pa 3085 574}%
\special{pa 3150 550}%
\special{fp}%
%
\special{pn 13}%
\special{ar 3450 700 150 150 3.1415927 4.7123890}%
%
\special{pn 13}%
\special{ar 3450 700 150 150 4.7123890 6.2831853}%
%
\special{pn 13}%
\special{pa 3465 551}%
\special{pa 3450 550}%
\special{fp}%
\special{sh 1}%
\special{pa 3450 550}%
\special{pa 3515 574}%
\special{pa 3503 554}%
\special{pa 3518 534}%
\special{pa 3450 550}%
\special{fp}%
\put(34.5000,-6.0000){\makebox(0,0){$-$}}%
\put(33.0000,-4.5000){\makebox(0,0){$-$}}%
\put(21.0000,-10.0000){\makebox(0,0){$A_D$}}%
\put(33.0000,-10.0000){\makebox(0,0){$G_D$}}%
%
\special{pn 13}%
\special{pa 2000 700}%
\special{pa 2200 700}%
\special{fp}%
\special{sh 1}%
\special{pa 2200 700}%
\special{pa 2133 680}%
\special{pa 2147 700}%
\special{pa 2133 720}%
\special{pa 2200 700}%
\special{fp}%
%
\special{pn 13}%
\special{pa 2300 700}%
\special{pa 2500 700}%
\special{fp}%
\special{sh 1}%
\special{pa 2500 700}%
\special{pa 2433 680}%
\special{pa 2447 700}%
\special{pa 2433 720}%
\special{pa 2500 700}%
\special{fp}%
%
\special{pn 13}%
\special{pa 3200 700}%
\special{pa 3400 700}%
\special{fp}%
\special{sh 1}%
\special{pa 3400 700}%
\special{pa 3333 680}%
\special{pa 3347 700}%
\special{pa 3333 720}%
\special{pa 3400 700}%
\special{fp}%
%
\special{pn 13}%
\special{pa 3500 700}%
\special{pa 3700 700}%
\special{fp}%
\special{sh 1}%
\special{pa 3700 700}%
\special{pa 3633 680}%
\special{pa 3647 700}%
\special{pa 3633 720}%
\special{pa 3700 700}%
\special{fp}%
%
\special{pn 13}%
\special{pa 400 700}%
\special{pa 650 700}%
\special{fp}%
\put(18.0000,-8.0000){\makebox(0,0){$f_1$}}%
\put(21.0000,-8.0000){\makebox(0,0){$f_2$}}%
\put(24.0000,-8.0000){\makebox(0,0){$f_3$}}%
\put(36.0000,-8.0000){\makebox(0,0){$f_3$}}%
\put(33.0000,-8.0000){\makebox(0,0){$f_2$}}%
\put(30.0000,-8.0000){\makebox(0,0){$f_1$}}%
\end{picture}}%
\caption{Examples of arrow / Gauss diagrams for triple crossings; the three arcs are contained in mutually distinct components of a 3-component link, and the arc of $f_1$ is front of that of $f_3$}
\label{fig:example_arrow_Gauss_3-strand}
\end{figure}
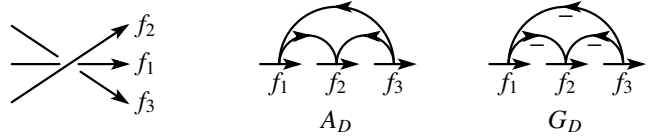

\begin{defn}\label{def:pairing_for_knots}
For a diagram $D$ of a long knot,
let 
$C^2(D)$ the set of unordered pairs $(\alpha,\beta)$ of chords $\alpha,\beta$ of $A_D$.
Let $X$ and $X'_k$, $k=1,2,3$, be the arrow diagrams as shown in Figure~\ref{fig:X_knot}, and define
\begin{equation}
 C^2_*(D)\coloneqq\{(\alpha,\beta)\in C^2(D)\mid\alpha\text{ and }\beta\text{ form a subdiagram of }A_D\text{ that is isomorphic to }*\},
 \quad
 *=X,X'_1,X'_2\text{ or }X'_3.
\end{equation}
Then define
\begin{equation}
 \pair{X}{G_D} \coloneqq\sum_{(\alpha,\beta)\in C^2_X(D)}\sgn(\alpha)\sgn(\beta),
 \quad
 \pair{X'_k}{G_D}\coloneqq\sum_{(\alpha,\beta)\in C^2_{X'_k}(D)}\sgn(\alpha)\sgn(\beta),
\end{equation}
where $\sgn(\alpha)\in\{+1,-1\}$ is the sign of $\alpha$ seen as a chord of $G_D$.
\end{defn}
\begin{figure}
\centering
{\unitlength 0.1in%
\begin{picture}(40.0000,4.2700)(4.0000,-8.2700)%
%
\special{pn 13}%
\special{pa 400 700}%
\special{pa 1200 700}%
\special{fp}%
\special{sh 1}%
\special{pa 1200 700}%
\special{pa 1133 680}%
\special{pa 1147 700}%
\special{pa 1133 720}%
\special{pa 1200 700}%
\special{fp}%
%
\special{pn 13}%
\special{ar 650 700 200 200 4.7123890 6.2831853}%
%
\special{pn 13}%
\special{ar 850 700 200 200 4.7123890 6.2831853}%
%
\special{pn 13}%
\special{pa 865 501}%
\special{pa 850 500}%
\special{fp}%
\special{sh 1}%
\special{pa 850 500}%
\special{pa 915 524}%
\special{pa 903 504}%
\special{pa 918 484}%
\special{pa 850 500}%
\special{fp}%
%
\special{pn 13}%
\special{ar 650 700 200 200 3.1415927 4.7123890}%
%
\special{pn 13}%
\special{pa 635 501}%
\special{pa 650 500}%
\special{fp}%
\special{sh 1}%
\special{pa 650 500}%
\special{pa 582 484}%
\special{pa 597 504}%
\special{pa 585 524}%
\special{pa 650 500}%
\special{fp}%
%
\special{pn 13}%
\special{ar 850 700 200 200 3.1415927 4.7123890}%
\put(7.5000,-9.0000){\makebox(0,0){$X$}}%
%
\special{pn 13}%
\special{pa 1600 700}%
\special{pa 2400 700}%
\special{fp}%
\special{sh 1}%
\special{pa 2400 700}%
\special{pa 2333 680}%
\special{pa 2347 700}%
\special{pa 2333 720}%
\special{pa 2400 700}%
\special{fp}%
%
\special{pn 13}%
\special{ar 1900 700 200 200 3.1415927 4.7123890}%
%
\special{pn 13}%
\special{pa 1885 501}%
\special{pa 1900 500}%
\special{fp}%
\special{sh 1}%
\special{pa 1900 500}%
\special{pa 1832 484}%
\special{pa 1847 504}%
\special{pa 1835 524}%
\special{pa 1900 500}%
\special{fp}%
%
\special{pn 13}%
\special{ar 1900 700 200 200 4.7123890 6.2831853}%
%
\special{pn 13}%
\special{ar 2000 700 300 300 4.7123890 6.2831853}%
%
\special{pn 13}%
\special{pa 2015 400}%
\special{pa 2000 400}%
\special{fp}%
\special{sh 1}%
\special{pa 2000 400}%
\special{pa 2067 420}%
\special{pa 2053 400}%
\special{pa 2067 380}%
\special{pa 2000 400}%
\special{fp}%
%
\special{pn 13}%
\special{ar 2000 700 300 300 3.1415927 4.7123890}%
\put(20.0000,-9.0000){\makebox(0,0){$X'_1$}}%
%
\special{pn 13}%
\special{pa 2600 700}%
\special{pa 3400 700}%
\special{fp}%
\special{sh 1}%
\special{pa 3400 700}%
\special{pa 3333 680}%
\special{pa 3347 700}%
\special{pa 3333 720}%
\special{pa 3400 700}%
\special{fp}%
%
\special{pn 13}%
\special{ar 2850 700 150 150 4.7123890 6.2831853}%
%
\special{pn 13}%
\special{ar 3150 700 150 150 4.7123890 6.2831853}%
%
\special{pn 13}%
\special{pa 3165 551}%
\special{pa 3150 550}%
\special{fp}%
\special{sh 1}%
\special{pa 3150 550}%
\special{pa 3215 574}%
\special{pa 3203 554}%
\special{pa 3218 534}%
\special{pa 3150 550}%
\special{fp}%
%
\special{pn 13}%
\special{ar 3150 700 150 150 3.1415927 4.7123890}%
%
\special{pn 13}%
\special{ar 2850 700 150 150 3.1415927 4.7123890}%
%
\special{pn 13}%
\special{pa 2835 551}%
\special{pa 2850 550}%
\special{fp}%
\special{sh 1}%
\special{pa 2850 550}%
\special{pa 2782 534}%
\special{pa 2797 554}%
\special{pa 2785 574}%
\special{pa 2850 550}%
\special{fp}%
\put(30.0000,-9.0000){\makebox(0,0){$X'_2$}}%
%
\special{pn 13}%
\special{pa 3600 700}%
\special{pa 4400 700}%
\special{fp}%
\special{sh 1}%
\special{pa 4400 700}%
\special{pa 4333 680}%
\special{pa 4347 700}%
\special{pa 4333 720}%
\special{pa 4400 700}%
\special{fp}%
%
\special{pn 13}%
\special{ar 4000 700 300 300 4.7123890 6.2831853}%
%
\special{pn 13}%
\special{ar 4000 700 300 300 3.1415927 4.7123890}%
%
\special{pn 13}%
\special{pa 3985 400}%
\special{pa 4000 400}%
\special{fp}%
\special{sh 1}%
\special{pa 4000 400}%
\special{pa 3933 380}%
\special{pa 3947 400}%
\special{pa 3933 420}%
\special{pa 4000 400}%
\special{fp}%
%
\special{pn 13}%
\special{ar 4100 700 200 200 3.1415927 4.7123890}%
%
\special{pn 13}%
\special{ar 4100 700 200 200 4.7123890 6.2831853}%
%
\special{pn 13}%
\special{pa 4115 501}%
\special{pa 4100 500}%
\special{fp}%
\special{sh 1}%
\special{pa 4100 500}%
\special{pa 4165 524}%
\special{pa 4153 504}%
\special{pa 4168 484}%
\special{pa 4100 500}%
\special{fp}%
\put(40.0000,-9.0000){\makebox(0,0){$X'_3$}}%
\end{picture}}%
\caption{Arrow diagrams $X$ and $X'_k$ ($k=1,2,3$)}
\label{fig:X_knot}
\end{figure}
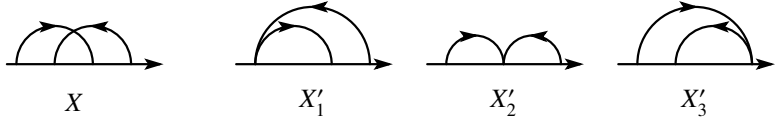

Let $c$ denote the \emph{Casson invariant} of (long) knots, namely the coefficient of the quadratic term of the Alexander-Conway polynomial.
Brooks and Komendarczyk \cite{BrooksKomendarczyk24} gave a combinatorial formula that computes the value of $c$ from (long) knot diagrams possibly with multiple crossings.

\begin{thm}[{\cite[Theorem A, B]{BrooksKomendarczyk24}, see also \cite[Theorem 1]{PolyakViro01}}]\label{thm:BrooksKomendarczyk}
If a diagram $D$ of a long knot $f$ has only transverse double crossing, then
\begin{equation}\label{eq:Polyak_formula_for_Casson}
 c(f)=\pair{X}{G_D}.
\end{equation}
For a general diagram $D$ of a long knot $f$, we have
\begin{equation}\label{eq:generalized_Polyak_formula_for_Casson}
 c(f)=\pair{X}{G_D}+\frac{1}{2}\left(\pair{X'_1}{G_D}+\pair{X'_2}{G_D}+\pair{X'_3}{G_D}\right).
\end{equation}
\end{thm}

Our main results are analogous formulas to \eqref{eq:Polyak_formula_for_Casson} and \eqref{eq:generalized_Polyak_formula_for_Casson} for the \emph{Milnor triple linking number} \cite{Milnor57}.

\begin{defn}
For a diagram $D$ of a 3-component long link $f$,
let $C^2(D)$ be defined in the same way as in Definition~\ref{def:pairing_for_knots}.
Let $\X_k$ and $\mathcal{X}'_k$, $k=1,2,3$, be the arrow diagrams as shown in Figure~\ref{fig:X_link}, and define
\begin{equation}
 C^2_*(D)\coloneqq\{(\alpha,\beta)\in C^2(D)\mid\alpha\text{ and }\beta\text{ form a subdiagram of }A_D\text{ that is isomorphic to }*\},
 \quad *=\X_k\text{ or }\X'_k\ (k=1,2,3).
\end{equation}
Then define
\begin{equation}
 \pair{\X_k}{G_D}\coloneqq\sum_{(\alpha,\beta)\in C^2_{\X_k}(D)}\sgn(\alpha)\sgn(\beta),
 \quad
 \pair{\X'_k}{G_D}\coloneqq\sum_{(\alpha,\beta)\in C^2_{\X'_k}(D)}\sgn(\alpha)\sgn(\beta).
\end{equation}
\end{defn}
\begin{figure}
\centering
\input{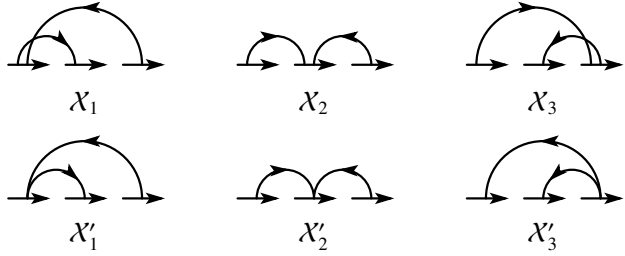}
\caption{Arrow diagrams $\X_k$ and $\X'_k$ ($k=1,2,3$)}
\label{fig:X_link}
\end{figure}

\begin{thm}\label{thm:main}
If a diagram $D$ of a 3-component long link $f$ has only transverse double crossing, then
\begin{equation}\label{eq:Polyak_formula_for_triple_linking}
 \mu(f)=\pair{\X_1}{G_D}+\pair{\X_2}{G_D}+\pair{\X_3}{G_D}.
\end{equation}
For a general diagram $D$ of a 3-component long link $f$, we have
\begin{equation}\label{eq:generalized_Polyak_formula_for_triple_linking}
 \mu(f)=\pair{\X_1}{G_D}+\pair{\X_2}{G_D}+\pair{\X_3}{G_D}+\frac{1}{2}\left(\pair{\X'_1}{G_D}+\pair{\X'_2}{G_D}+\pair{\X'_3}{G_D}\right).
\end{equation}
\end{thm}

\begin{rem}
\eqref{eq:Polyak_formula_for_triple_linking} has already been known as \cite[Theorem 3]{Polyak97}, proved in a different way from ours.
\end{rem}

\section{Configuration space integral and the Polyak formula for the Milnor triple linking number}\label{s:CSI}
Configuration space integral is firstly introduced in \cite{BottTaubes94,Kohno94} to produce all the Vassiliev invariants for knots from appropriate formal sums of trivalent graphs,
and then generalized in \cite{CCL02,KoytcheffMunsonVolic13} to give cochain maps (up to some correction terms) from graph complexes to de Rham complex of the spaces of (long) knots and links.
In this paper we only give the constructions of the linking number, the Casson knot invariant and the Milnor triple linking number.
See also \cite{AltschulerFreidel97,Koytcheff14,Volic05} for general construction.

\subsection{The linking number of 2-component links}
For simplicity, we treat with ordinary 2-component links, namely embeddings $f\colon S^1\sqcup S^1\hookrightarrow\R^3$.

For a 2-component link $f=(f_1,f_2)$,
define a map
\begin{equation}
 h=h_f\colon S^1\times S^1\to S^2,
 \quad
 h(x_1,x_2)\coloneqq\frac{f_2(x_2)-f_1(x_1)}{\abs{f_2(x_2)-f_1(x_1)}}.
\end{equation}
Then the linking number of $f$, the mapping degree of $h$,
is given by
\begin{equation}\label{eq:linking_number}
 \link(f)\coloneqq\int_{S^1\times S^1}h^*\eta,
\end{equation}
here $\eta$ is a ``Dirac-like'' $2$-form $\eta=\eta^N_{\epsilon}\in\Omega^2_{dR}(S^2)$ that satisfies
\begin{enumerate}[(i)]
\item
	$\supp(\eta)$ is contained in the $\epsilon$-neighborhood of the north pole $N=(0,0,1)\in\R^3$,
\item
	$\eta$ is invariant under the action of $\O(2)=\O(2)\times\{\id\}\subset\O(3)$ (namely $\varphi^*\eta=(-1)^{\det\varphi}\eta$ for $\varphi\in\O(2)$), and
\item
	$\eta$ is a unit 2-form, that is, $\displaystyle\int_{S^2}\eta=1$.
\end{enumerate}
The construction \eqref{eq:linking_number} is encoded by the graph shown in Figure~\ref{fig:graph_linking_number};
\begin{figure}
\centering
{\unitlength 0.1in%
\begin{picture}(16.0000,6.0000)(3.0000,-8.0000)%
%
\special{pn 13}%
\special{ar 600 500 300 300 0.0000000 6.2831853}%
%
\special{pn 13}%
\special{ar 1600 500 300 300 0.0000000 6.2831853}%
%
\special{pn 13}%
\special{pa 900 500}%
\special{pa 1300 500}%
\special{fp}%
%
\special{sh 1.000}%
\special{ia 900 500 25 25 0.0000000 6.2831853}%
\special{pn 13}%
\special{ar 900 500 25 25 0.0000000 6.2831853}%
%
\special{sh 1.000}%
\special{ia 1300 500 25 25 0.0000000 6.2831853}%
\special{pn 13}%
\special{ar 1300 500 25 25 0.0000000 6.2831853}%
\end{picture}}%
\caption{The graph that corresponds to the linking number}
\label{fig:graph_linking_number}
\end{figure}
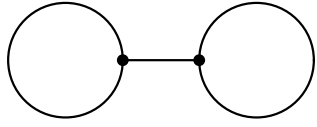
we consider $S^1\times S^1$, thought of as ``two-point configuration space on $S^1$,'' associated to the two vertices on the circles (one on each component),
and also consider the map $h$ associated to the edge connecting these two vertices.

The integral \eqref{eq:linking_number} is independent of the choice of non-exact unit 2-forms of $S^2$.
Following \cite{BrooksKomendarczyk24} we choose $\eta$, that simplifies computations of the integrals as we see in Example~\ref{ex:linking_number} below.
In fact, we do not need to impose the assumption (ii) to define the linking number;
we need it only for the Milnor triple linking number.

\begin{ex}\label{ex:linking_number}
Let $f=(f_1,f_2)$ be the Hopf link (see Figure~\ref{fig:Hopf}), and suppose a diagram $D$ of $f$ has exactly two crossings $p$ and $q$ with $f_2$ being the over-arc at $p=\pr_{12}(f_1(t))=\pr_{12}(f_2(u))$.
If the radius $\epsilon>0$ of $\supp(\eta)$ is sufficiently close to $0$,
then $h(x_1,x_2)\in\supp(\eta)$ only if $x_1$ and $x_2$ are near $t$ and $u$ respectively.
Thus there exist neighborhoods $U$ and $V$ of $t$ and $u$ such that
\begin{equation}\label{eq:linking_number_localized}
 \link(f_1,f_2)=\int_{U\times V}h^*\eta.
\end{equation}
Since $\link(f_1,f_2)$ is equal to the sign of $p$, we have
\begin{equation}\label{eq:localized_integral}
 \int_{U\times V}h^*\eta=\sgn(p).
\end{equation}
\end{ex}
\begin{figure}
\centering
{\unitlength 0.1in%
\begin{picture}(32.2500,10.0000)(4.0000,-12.2700)%
\put(5.0000,-10.0000){\makebox(0,0)[lb]{$f_1$}}%
\put(15.0000,-10.0000){\makebox(0,0)[rb]{$f_2$}}%
\put(10.0000,-13.0000){\makebox(0,0){$p$}}%
\put(10.0000,-3.0000){\makebox(0,0){$q$}}%
%
\special{pn 8}%
\special{pa 2800 1200}%
\special{pa 3150 850}%
\special{fp}%
\put(28.0000,-12.0000){\makebox(0,0)[rb]{$f_1$}}%
\put(36.2500,-12.0000){\makebox(0,0)[lb]{$f_2$}}%
%
\special{pn 20}%
\special{pa 2900 1100}%
\special{pa 3150 850}%
\special{fp}%
%
\special{pn 20}%
\special{pa 3250 750}%
\special{pa 3500 500}%
\special{fp}%
%
\special{pn 20}%
\special{pa 2900 500}%
\special{pa 3500 1100}%
\special{fp}%
\put(35.0000,-5.2500){\makebox(0,0)[lt]{$f_1(U)$}}%
\put(29.0000,-5.2500){\makebox(0,0)[rt]{$f_2(V)$}}%
\put(31.0000,-12.0000){\makebox(0,0)[lt]{$f_1(x_1)$}}%
%
\special{sh 1.000}%
\special{ia 3000 1000 25 25 0.0000000 6.2831853}%
\special{pn 8}%
\special{ar 3000 1000 25 25 0.0000000 6.2831853}%
%
\special{sh 1.000}%
\special{ia 3100 700 25 25 0.0000000 6.2831853}%
\special{pn 8}%
\special{ar 3100 700 25 25 0.0000000 6.2831853}%
%
\special{pn 4}%
\special{pa 3100 1200}%
\special{pa 3000 1000}%
\special{fp}%
\put(32.0000,-4.0000){\makebox(0,0)[lb]{$f_2(x_2)$}}%
%
\special{pn 4}%
\special{pa 3200 400}%
\special{pa 3100 700}%
\special{fp}%
%
\special{pn 13}%
\special{pa 3000 1000}%
\special{pa 3060 820}%
\special{dt 0.030}%
\special{sh 1}%
\special{pa 3060 820}%
\special{pa 3020 877}%
\special{pa 3043 871}%
\special{pa 3058 890}%
\special{pa 3060 820}%
\special{fp}%
\put(30.0000,-9.0000){\makebox(0,0)[rb]{$h$}}%
%
\special{pn 13}%
\special{pa 3060 820}%
\special{pa 3100 700}%
\special{dt 0.030}%
\put(32.0000,-10.0000){\makebox(0,0){$p$}}%
%
\special{pn 13}%
\special{ar 800 800 400 400 1.1071487 0.9272952}%
%
\special{pn 13}%
\special{ar 1200 800 400 400 4.2487414 4.0688879}%
%
\special{pn 8}%
\special{pa 3250 750}%
\special{pa 3600 400}%
\special{fp}%
%
\special{pn 8}%
\special{pa 2800 400}%
\special{pa 3600 1200}%
\special{fp}%
\end{picture}}%
\caption{The Hopf link}
\label{fig:Hopf}
\end{figure}
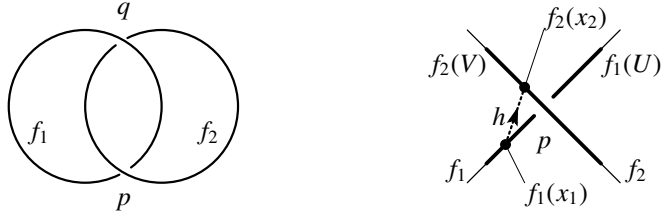

Hereafter we consider similar integrals to \eqref{eq:linking_number} associated to other arrow diagrams or graphs.
Since we use $\eta$,
in many cases such integrals ``localize'' to those near the crossing points,
and then \eqref{eq:localized_integral} converts the computation of the integrals to a combinatorics, namely a counting of certain kinds of crossings.

\subsection{The Casson invariant for long knots}
Let $f$ be a long knot.
Associated to $X$ (Figure~\ref{fig:X_knot}),
consider the configuration space
\begin{equation}
 C_X\coloneqq\{(x_1,\dots,x_4)\in(\R^1)^{\times 4}\mid x_1<\dots<x_4\}
\end{equation}
and the maps $h_{13},h_{42}\colon C_X\to S^2$ defined by
\begin{equation}\label{eq:h_X}
 h_{ij}(x_1,\dots,x_4)\coloneqq\frac{f(x_j)-f(x_i)}{\abs{f(x_j)-f(x_i)}},
 \quad
 (i,j)=(1,3),(4,2).
\end{equation}
We remark that these maps are obviously defined as indicated by the oriented chords of $X$ (Figure~\ref{fig:X_knot}).
We put $h_X\coloneqq h_{13}\times h_{42}\colon C_X\to(S^2)^{\times 2}$ for simplicity, and consider the (possibly diverging) integral
\begin{equation}\label{eq:I_X_knot}
 I_X(f)\coloneqq\int_{C_X}\omega_X,
 \quad
 \omega_{X}\coloneqq h^*_X(\eta\times\eta).
\end{equation}
Next,
associated to the graph $Y$ (Figure~\ref{fig:Y_link}, the left),
consider the configuration space
\begin{equation}
 C_Y(f)\coloneqq\{(x_1,x_2,x_3;x_4)\in(\R^1)^{\times 2}\times\R^3\mid f(x_i)\ne x_4,\ i=1,2,3\},
\end{equation}
and the maps $h_{i4}\colon C_Y(f)\to S^2$ ($i=1,3$) and $h_{42}\colon C_Y(f)\to S^2$ defined by
\begin{equation}
 h_{i4}(x_1,x_2,x_3;x_4)\coloneqq\frac{x_4-f(x_i)}{\abs{x_4-f(x_i)}}
 \quad(i=1,3),
 \quad
 h_{42}(x_1,x_2,x_3;x_4)\coloneqq\frac{f(x_2)-x_4}{\abs{f(x_2)-x_4}}.
\end{equation}
\begin{figure}
\centering
{\unitlength 0.1in%
\begin{picture}(8.0000,3.2100)(4.0000,-6.2100)%
%
\special{pn 13}%
\special{pa 400 600}%
\special{pa 1200 600}%
\special{fp}%
\special{sh 1}%
\special{pa 1200 600}%
\special{pa 1133 580}%
\special{pa 1147 600}%
\special{pa 1133 620}%
\special{pa 1200 600}%
\special{fp}%
%
\special{pn 13}%
\special{pa 800 300}%
\special{pa 800 500}%
\special{fp}%
\special{sh 1}%
\special{pa 800 500}%
\special{pa 820 433}%
\special{pa 800 447}%
\special{pa 780 433}%
\special{pa 800 500}%
\special{fp}%
%
\special{pn 13}%
\special{pa 500 600}%
\special{pa 700 400}%
\special{fp}%
\special{sh 1}%
\special{pa 700 400}%
\special{pa 639 433}%
\special{pa 662 438}%
\special{pa 667 461}%
\special{pa 700 400}%
\special{fp}%
%
\special{pn 13}%
\special{pa 1100 600}%
\special{pa 900 400}%
\special{fp}%
\special{sh 1}%
\special{pa 900 400}%
\special{pa 933 461}%
\special{pa 938 438}%
\special{pa 961 433}%
\special{pa 900 400}%
\special{fp}%
%
\special{pn 13}%
\special{pa 900 400}%
\special{pa 800 300}%
\special{fp}%
%
\special{pn 13}%
\special{pa 700 400}%
\special{pa 800 300}%
\special{fp}%
%
\special{pn 13}%
\special{pa 800 500}%
\special{pa 800 600}%
\special{fp}%
\end{picture}}%
\qquad
{\unitlength 0.1in%
\begin{picture}(8.0000,3.2100)(4.0000,-6.2100)%
%
\special{pn 13}%
\special{pa 400 600}%
\special{pa 600 600}%
\special{fp}%
\special{sh 1}%
\special{pa 600 600}%
\special{pa 533 580}%
\special{pa 547 600}%
\special{pa 533 620}%
\special{pa 600 600}%
\special{fp}%
%
\special{pn 13}%
\special{pa 800 300}%
\special{pa 800 500}%
\special{fp}%
\special{sh 1}%
\special{pa 800 500}%
\special{pa 820 433}%
\special{pa 800 447}%
\special{pa 780 433}%
\special{pa 800 500}%
\special{fp}%
%
\special{pn 13}%
\special{pa 500 600}%
\special{pa 700 400}%
\special{fp}%
\special{sh 1}%
\special{pa 700 400}%
\special{pa 639 433}%
\special{pa 662 438}%
\special{pa 667 461}%
\special{pa 700 400}%
\special{fp}%
%
\special{pn 13}%
\special{pa 1100 600}%
\special{pa 900 400}%
\special{fp}%
\special{sh 1}%
\special{pa 900 400}%
\special{pa 933 461}%
\special{pa 938 438}%
\special{pa 961 433}%
\special{pa 900 400}%
\special{fp}%
%
\special{pn 13}%
\special{pa 900 400}%
\special{pa 800 300}%
\special{fp}%
%
\special{pn 13}%
\special{pa 700 400}%
\special{pa 800 300}%
\special{fp}%
%
\special{pn 13}%
\special{pa 800 500}%
\special{pa 800 600}%
\special{fp}%
%
\special{pn 13}%
\special{pa 700 600}%
\special{pa 900 600}%
\special{fp}%
\special{sh 1}%
\special{pa 900 600}%
\special{pa 833 580}%
\special{pa 847 600}%
\special{pa 833 620}%
\special{pa 900 600}%
\special{fp}%
%
\special{pn 13}%
\special{pa 1000 600}%
\special{pa 1200 600}%
\special{fp}%
\special{sh 1}%
\special{pa 1200 600}%
\special{pa 1133 580}%
\special{pa 1147 600}%
\special{pa 1133 620}%
\special{pa 1200 600}%
\special{fp}%
\end{picture}}%
\caption{The graphs $Y$ and $\Y$}
\label{fig:Y_link}
\end{figure}
These maps are also defined as indicated by the oriented chords of $Y$.
We put $h_Y\coloneqq h_{14}\times h_{42}\times h_{34}\colon C_Y(f)\to(S^2)^{\times 3}$ and define
\begin{equation}\label{eq:I_Y_knot}
 I_Y(f)\coloneqq\int_{C_Y(f)}\omega_Y,
 \quad\text{where}\quad
 \omega_Y\coloneqq h^*_Y(\eta\times\eta\times\eta).
\end{equation}

\begin{prop}\label{prop:convergence_knot}
The integrals $I_X(f)$ and $I_Y(f)$ converge.
\end{prop}
\begin{proof}
The Axelrod-Singer compactifications \cite{AxelrodSinger94} of $C_X(f)$ and $C_Y(f)$ are manifolds with corners, onto which the maps $h_{ij}$ are smoothly extended.
Thus $\omega_X$ and $\omega_Y$ are defined on the compactified configuration spaces, and their integrals converge.
\end{proof}

\begin{thm}[{\cite{BottTaubes94,Kohno94,BrooksKomendarczyk24}}]\label{thm:CSI_Casson}
$I_X(f)-I_Y(f)$ is an invariant for long knots and is equal to the Casson invariant $c(f)$ of $f$.
\end{thm}

Theorem~\ref{thm:CSI_Casson} can be proved in the same way as in \cite{BottTaubes94,Kohno94}, but we have to pay attention to the choice of the volume form $\eta$.

Indeed, \eqref{eq:I_X_knot} and \eqref{eq:I_Y_knot} give rise to $0$-forms on the space of long knots.
By the Stokes' theorem (that is generalized so that it can be applied to integrations along the fiber), the exterior differentials of the $0$-forms $I_*$ ($*=X,Y$) are
\begin{equation}
 dI_*=\int_{C_*}d\omega_*\pm\int_{\partial C_*}\omega_*=\pm\int_{\partial C_*}\omega_*.
\end{equation}
The boundary of (compactified) $C_X$ is stratified, and each codimension one (open) stratum consists of the configurations $(x_1,x_2,x_3,x_4)$ such that
\begin{enumerate}[(i)]
\item
	exactly two $x_i$'s collide,
\item
	three or more $x_i$'s collide, or
\item
	some of $x_i$ ``escape to infinity'' (or ``collide with the point at infinity'').
\end{enumerate}
It can be shown that the integrations over the strata of types (ii) and (iii) vanish, by dimensional reasons (see for example \cite[Appendix A]{CCL02}).
Thus $dI_X$ is equal to the sum of integrals over the type (i) strata.
The situation is similar for $Y$, and the type (i) contributions for $X$ and $Y$ cancel to each other because the formal sum $X-Y$ is a ``graph cocycle'' with respect to the coboundary map defined as the signed sum of the edge-contractions.
Thus we show that $c\coloneqq I_X-I_Y$ is a closed $0$-form on the space of long knots, that is, an isotopy invariant.
Indeed $c$ is the Casson invariant, because $c$ satisfies \eqref{eq:Polyak_formula_for_Casson} that is equal to the formula for the Casson invariant given in \cite{PolyakViro01}.

For general graphs $Z$ other than $X$ and $Y$, the boundary of the (compactified) $C_Z$ has other types of strata, where three or more points collide.
If we use a volume form that is invariant with respect to the antipodal map $S^2\to S^2$, then almost all the integrals over such strata turns out to vanish, by a ``symmetry argument'' (see for example \cite[Appendix A]{CCL02}).
Our $\eta$ is not symmetric and these arguments might fail.
Moreover there exists one more type of strata, the \emph{anomalous strata}, where all the points collide simultaneously.
It is not known whether the integrals over the anomalous strata vanish or not for general graphs, and we need to add some correction terms.
Fortunately, for both $X$ and $Y$, such symmetry arguments as above are not necessary, and the contribution of the anomalous strata can be shown to vanish by dimensional reasons.

\subsection{The Milnor triple linking number}
Let $f=f_1\sqcup f_2\sqcup f_3$ be a 3-component long link.
Associated to $\X_1$ (Figure~\ref{fig:X_link}),
consider the configuration space
\begin{equation}
 C_{\X_1}\coloneqq\{((x_1,x_2),x_3,x_4)\in(\R^1)^{\times 2}\times\R^1\times\R^1\mid x_1<x_2\}.
\end{equation}
We think of $((x_1,x_2),x_3,x_4)\in C_{\X_1}$ as a quadruple of points on $f$, with $x_1,x_2$ on $f_1$, $x_3$ on $f_2$ and $x_4$ on $f_3$.
Define the maps $h_{13},h_{42}\colon C_{\X_1}\to S^2$ as indicated by the oriented chords of $\X_1$ (Figure~\ref{fig:X_link});
\begin{equation}\label{eq:maps_for_X1}
 h_{13}(x_1,\dots,x_4)\coloneqq\frac{f_2(x_3)-f_1(x_1)}{\abs{f_2(x_3)-f_1(x_1)}},
 \quad
 h_{42}(x_1,\dots,x_4)\coloneqq\frac{f_3(x_4)-f_1(x_2)}{\abs{f_3(x_4)-f_1(x_2)}}.
\end{equation}
We put $h_{\X_1}\coloneqq h_{13}\times h_{42}\colon C_{\X_1}\to(S^2)^{\times 2}$ and define
\begin{equation}\label{eq:I_X_1}
 I_{\X_1}(f)\coloneqq\int_{C_{\X_1}}\omega_{\X_1},
 \quad\text{where}\quad
 \omega_{\X_1}\coloneqq h_{\X_1}^*(\eta\times\eta).
\end{equation}
Similarly,
associated to $\X_2$ and $\X_3$ (Figure~\ref{fig:X_link}),
consider the configuration spaces
\begin{align}
 C_{\X_2}&\coloneqq\{(x_1,(x_2,x_3),x_4)\in\R^1\times(\R^1)^{\times 2}\times\R^1\mid x_2<x_3\},\\
 C_{\X_3}&\coloneqq\{(x_1,x_2,(x_3,x_4))\in\R^1\times\R^1\times(\R^1)^{\times 2}\mid x_3<x_4\},
\end{align}
and define the maps $h_{12},h_{43}\colon C_{\X_2}\to S^2$ and $h_{13},h_{42}\colon C_{\X_3}\to S^2$ as indicated by the oriented chords of $\X_2$ and $\X_3$ (Figure~\ref{fig:X_link}).
We put
$h_{\X_2}\coloneqq h_{12}\times h_{43}\colon C_{\X_2}\to(S^2)^{\times 2}$
and
$h_{\X_3}\coloneqq h_{13}\times h_{42}\colon C_{\X_3}\to(S^2)^{\times 2}$,
and define
\begin{equation}\label{eq:I_X_23}
 I_{\X_k}(f)\coloneqq\int_{C_{\X_k}}\omega_{\X_k},
 \quad\text{where}\quad
 \omega_{\X_k}\coloneqq h^*_{\X_k}(\eta\times\eta)
 \quad(k=2,3).
\end{equation}
Finally,
associated to the graph $\Y$ (Figure \ref{fig:Y_link}, the right),
consider the configuration space
\begin{equation}
 C_{\Y}(f)\coloneqq\{(x_1,x_2,x_3;x_4)\in(\R^1)^{\times 3}\times\R^3\mid f_k(x_k)\ne x_4\ (k=1,2,3)\}
\end{equation}
and define the maps $h_{i4}\colon C_{\Y}(f)\to S^2$ ($i=1,3$) and $h_{42}\colon C_{\Y}(f)\to S^2$ as indicated by the oriented chords of $\Y$ (Figure~\ref{fig:Y_link}).
We put $h_{\Y}\coloneqq h_{14}\times h_{42}\times h_{34}\colon C_{\Y}(f)\to(S^2)^{\times 3}$ and define
\begin{equation}\label{eq:I_Y}
 I_{\Y}(f)\coloneqq\int_{C_{\Y}(f)}\omega_{\Y},
 \quad\text{where}\quad
 \omega_{\Y}\coloneqq h^*_{\Y}(\eta\times\eta\times\eta).
\end{equation}

\begin{prop}\label{prop:convergence_link}
The integrals $I_{\X_k}(f)$ ($k=1,2,3$) and $I_{\Y}(f)$ converge.
\end{prop}

This is because, similar to Proposition~\ref{prop:convergence_knot}, the Axelrod-Singer compactifications \cite{AxelrodSinger94} for $C_{\X_k}$ ($k=1,2,3$) and $C_{\Y}$ also exist.

\begin{thm}[\cite{Koytcheff14}]\label{thm:CSI_triple_linking}
$\mu(f)\coloneqq I_{\X_1}(f)+I_{\X_2}(f)+I_{\X_3}(f)-I_{\Y}(f)$ is equal to the Milnor triple linking number $\mu_{123}(f)$ of $f$.
\end{thm}


The proof of Theorem~\ref{thm:CSI_triple_linking} is the same as that in \cite{Koytcheff14}.
Similarly to the case of the Casson invariant,
we need to check that no symmetry argument is needed.
The proof is easier in this case than in the case of the Casson invariant; there exist no anomalous faces for $\X_k$'s and $\Y$.
This is because not all the four configuration points are on a single component and hence not all of them can collide.

\subsection{Proof of \eqref{eq:Polyak_formula_for_triple_linking}, the first half of Theorem~\ref{thm:main}}\label{ss:double_contribution_X}
The Polyak formulas \eqref{eq:Polyak_formula_for_Casson} and \eqref{eq:Polyak_formula_for_triple_linking} can be proved in parallel ways,
so we only give a proof of \eqref{eq:Polyak_formula_for_triple_linking}.
The proof below almost repeats verbatim that for \eqref{eq:Polyak_formula_for_Casson} in \cite{BrooksKomendarczyk24}.

Let us compute $I_{\X_1}(f)$ for a 3-component long link $f$ whose projection onto $N^{\perp}$ is a usual link diagram $D$, namely $D$ has only transeverse double crossings.

Let $p_1,p_2,\dots$ be the crossings of $D$ that involve $f_1$ and $f_2$ with arcs of $f_2$ as the over-arcs.
Similarly let $q_1,q_2,\dots$ be the crossings of $D$ that involve $f_1$ and $f_3$ with arcs of $f_1$ as their over-arcs.
Suppose $p_i=\pr_{12}(f_1(t_i))=\pr_{12}(f_3(u_i))$ and $q_j=\pr_{12}(f_1(t'_j))=\pr_{12}(f_2(u'_j))$.
Then, if $\epsilon$ is sufficiently close to $0$, the image $h_{\X_1}((x_1,x_2),x_3,x_4)$ of $((x_1,x_2),x_3,x_4)\in C_{\X_1}$ via the map $h_{\X_1}$ (the product of the maps in \eqref{eq:maps_for_X1}) is in $\supp(\eta)^{\times 2}$ only if $x_1$ and $x_3$ are near respectively $t_i$ and $u_i$ for some $i$, and $x_2$ and $x_4$ are near respectively $t'_j$ and $u'_j$ for some $j$, with $t_i<t'_j$ (Figure~\ref{fig:localized_X}).
\begin{figure}
\centering
{\unitlength 0.1in%
\begin{picture}(41.0000,8.7200)(2.0000,-10.2700)%
%
\special{pn 13}%
\special{pa 200 600}%
\special{pa 500 600}%
\special{fp}%
\special{sh 1}%
\special{pa 500 600}%
\special{pa 433 580}%
\special{pa 447 600}%
\special{pa 433 620}%
\special{pa 500 600}%
\special{fp}%
%
\special{pn 13}%
\special{pa 700 600}%
\special{pa 1000 600}%
\special{fp}%
\special{sh 1}%
\special{pa 1000 600}%
\special{pa 933 580}%
\special{pa 947 600}%
\special{pa 933 620}%
\special{pa 1000 600}%
\special{fp}%
%
\special{pn 13}%
\special{pa 1200 600}%
\special{pa 1500 600}%
\special{fp}%
\special{sh 1}%
\special{pa 1500 600}%
\special{pa 1433 580}%
\special{pa 1447 600}%
\special{pa 1433 620}%
\special{pa 1500 600}%
\special{fp}%
%
\special{pn 13}%
\special{ar 500 600 250 250 4.7123890 6.2831853}%
%
\special{pn 13}%
\special{ar 900 600 500 300 4.7123890 6.2831853}%
%
\special{pn 13}%
\special{pa 919 300}%
\special{pa 900 300}%
\special{fp}%
\special{sh 1}%
\special{pa 900 300}%
\special{pa 967 320}%
\special{pa 953 300}%
\special{pa 967 280}%
\special{pa 900 300}%
\special{fp}%
%
\special{pn 13}%
\special{ar 500 600 250 250 3.1415927 4.7123890}%
%
\special{pn 13}%
\special{pa 485 350}%
\special{pa 500 350}%
\special{fp}%
\special{sh 1}%
\special{pa 500 350}%
\special{pa 433 330}%
\special{pa 447 350}%
\special{pa 433 370}%
\special{pa 500 350}%
\special{fp}%
%
\special{pn 13}%
\special{ar 900 600 500 300 3.1415927 4.7123890}%
\put(8.5000,-9.0000){\makebox(0,0){$\X_1$}}%
%
\special{sh 1.000}%
\special{ia 2400 825 25 25 0.0000000 6.2831853}%
\special{pn 13}%
\special{ar 2400 825 25 25 0.0000000 6.2831853}%
%
\special{pn 13}%
\special{pa 1900 400}%
\special{pa 2700 200}%
\special{fp}%
%
\special{pn 13}%
\special{pa 1900 700}%
\special{pa 2700 900}%
\special{fp}%
%
\special{pn 8}%
\special{pa 2300 300}%
\special{pa 2300 1000}%
\special{dt 0.030}%
%
\special{sh 1.000}%
\special{ia 2200 325 25 25 0.0000000 6.2831853}%
\special{pn 13}%
\special{ar 2200 325 25 25 0.0000000 6.2831853}%
\put(27.0000,-9.0000){\makebox(0,0)[lt]{$f_1$}}%
\put(27.0000,-2.0000){\makebox(0,0)[lt]{$f_2$}}%
%
\special{sh 1.000}%
\special{ia 2300 1000 25 25 0.0000000 6.2831853}%
\special{pn 13}%
\special{ar 2300 1000 25 25 0.0000000 6.2831853}%
\put(23.0000,-11.0000){\makebox(0,0){$p_i$}}%
%
\special{pn 13}%
\special{pa 2400 825}%
\special{pa 2300 575}%
\special{dt 0.030}%
\special{sh 1}%
\special{pa 2300 575}%
\special{pa 2306 644}%
\special{pa 2320 625}%
\special{pa 2343 629}%
\special{pa 2300 575}%
\special{fp}%
\put(22.7500,-5.0000){\makebox(0,0)[rt]{$h_{13}$}}%
\put(24.2500,-8.0000){\makebox(0,0)[lb]{$f_1(x_1)$}}%
\put(21.7500,-3.0000){\makebox(0,0)[rb]{$f_2(x_3)$}}%
\put(2.5000,-7.0000){\makebox(0,0){1}}%
\put(4.0000,-7.0000){\makebox(0,0){2}}%
\put(7.5000,-7.0000){\makebox(0,0){3}}%
\put(14.0000,-7.0000){\makebox(0,0){4}}%
%
\special{pn 13}%
\special{pa 2300 575}%
\special{pa 2200 325}%
\special{dt 0.030}%
%
\special{pn 13}%
\special{pa 3500 400}%
\special{pa 4300 200}%
\special{fp}%
%
\special{pn 13}%
\special{pa 3500 700}%
\special{pa 4300 900}%
\special{fp}%
%
\special{pn 8}%
\special{pa 3900 300}%
\special{pa 3900 1000}%
\special{dt 0.030}%
%
\special{sh 1.000}%
\special{ia 3900 1000 25 25 0.0000000 6.2831853}%
\special{pn 13}%
\special{ar 3900 1000 25 25 0.0000000 6.2831853}%
\put(39.0000,-11.0000){\makebox(0,0){$q_j$}}%
%
\special{sh 1.000}%
\special{ia 3800 775 25 25 0.0000000 6.2831853}%
\special{pn 13}%
\special{ar 3800 775 25 25 0.0000000 6.2831853}%
%
\special{sh 1.000}%
\special{ia 3800 325 25 25 0.0000000 6.2831853}%
\special{pn 13}%
\special{ar 3800 325 25 25 0.0000000 6.2831853}%
%
\special{pn 13}%
\special{pa 3800 750}%
\special{pa 3800 500}%
\special{dt 0.030}%
\special{sh 1}%
\special{pa 3800 500}%
\special{pa 3780 567}%
\special{pa 3800 553}%
\special{pa 3820 567}%
\special{pa 3800 500}%
\special{fp}%
%
\special{pn 13}%
\special{pa 3800 500}%
\special{pa 3800 350}%
\special{dt 0.030}%
\put(43.0000,-2.0000){\makebox(0,0)[lt]{$f_1$}}%
\put(43.0000,-9.0000){\makebox(0,0)[lt]{$f_3$}}%
\put(37.7500,-3.0000){\makebox(0,0)[rb]{$f_1(x_2)$}}%
\put(37.7500,-8.0000){\makebox(0,0)[rt]{$f_3(x_4)$}}%
\put(36.5000,-5.5000){\makebox(0,0){$h_{42}$}}%
\end{picture}}%
\caption{}
\label{fig:localized_X}
\end{figure}
Thus, there exist disjoint neighborhoods $U_i^1$, $U_i^3$, $V_j^2$ and $V_j^4$ of respectively $t_i$, $u_i$, $t'_j$ and $u'_j$ such that
\begin{equation}\label{eq:I_X_1_localized_to_doublepoints}
 \lim_{\epsilon\to 0}I_{\X_1}(f)
 =\sum_{\substack{i,j \\ t_i<t'_j}}\int_{U_i^1\times V_j^2\times U_i^3\times V_j^4}\omega_{\X_1}
 =\sum_{\substack{i,j \\ t_i<t'_j}}\int_{U_i^1\times U_i^3}h^*_{13}\eta\int_{V_j^4\times V_j^2}h^*_{42}\eta,
\end{equation}
here we permute the coordinates as $(x_2,x_3,x_4)\mapsto(x_3,x_4,x_2)$, whose sign is $+1$.
The integrations in the most right hand side of \eqref{eq:I_X_1_localized_to_doublepoints} are nothing but the right hand side of \eqref{eq:linking_number_localized} in Example~\ref{ex:linking_number},
and are equal to the signs of the crossings $p_i$ and $q_j$.
The chords corresponding to $p_i$ and $q_j$ with $t_i<t'_j$ obviously form a subdiagram of $G_D$ isomorphic to $\X_1$.
Thus
\begin{equation}\label{eq:I_X_1_doublepoint_limit}
 \lim_{\epsilon\to 0}I_{\X_1}(f)=\sum_{\substack{i,j \\ t_i<t'_j}}\sgn(p_i)\sgn(q_j)=\pair{\X_1}{G_D}.
\end{equation}
In the same way we see that
\begin{equation}\label{eq:I_X_23_doublepoint_limit}
 \lim_{\epsilon\to 0}I_{\X_k}(f)=\pair{\X_k}{G_D},
 \quad
 k=2,3.
\end{equation}
As for $\Y$, the image of $(x_1,x_2,x_3;x_4)\in C_{\Y}(f)$ via $h_{\Y}$ is in $(\supp(\eta))^{\times 3}$ only if
\begin{itemize}
\item
	$\pr_{12}(f_i(x_i))$ ($i=1,2,3$) and $\pr_{12}(x_4)$ are sufficiently near, and
\item
	$f_i(x_i)$ sits ``below'' (resp.~``above'') $x_4$ if $i=1,3$ (resp.~$i=2$)
\end{itemize}
(see Figure~\ref{fig:localized_Y}).
\begin{figure}
\centering
{\unitlength 0.1in%
\begin{picture}(23.0000,8.7300)(0.0000,-11.0000)%
%
\special{pn 13}%
\special{pa 1500 400}%
\special{pa 2300 400}%
\special{fp}%
%
\special{pn 13}%
\special{pa 1500 600}%
\special{pa 2300 800}%
\special{fp}%
%
\special{pn 13}%
\special{pa 1500 1100}%
\special{pa 2300 900}%
\special{fp}%
\put(23.0000,-4.0000){\makebox(0,0)[lb]{$f_2$}}%
\put(23.0000,-8.0000){\makebox(0,0)[lb]{$f_1$}}%
\put(23.0000,-9.0000){\makebox(0,0)[lt]{$f_3$}}%
%
\special{sh 1.000}%
\special{ia 1900 400 25 25 0.0000000 6.2831853}%
\special{pn 8}%
\special{ar 1900 400 25 25 0.0000000 6.2831853}%
%
\special{sh 1.000}%
\special{ia 2000 600 25 25 0.0000000 6.2831853}%
\special{pn 8}%
\special{ar 2000 600 25 25 0.0000000 6.2831853}%
%
\special{sh 1.000}%
\special{ia 1900 1000 25 25 0.0000000 6.2831853}%
\special{pn 8}%
\special{ar 1900 1000 25 25 0.0000000 6.2831853}%
%
\special{sh 1.000}%
\special{ia 2100 750 25 25 0.0000000 6.2831853}%
\special{pn 8}%
\special{ar 2100 750 25 25 0.0000000 6.2831853}%
%
\special{pn 13}%
\special{pa 1990 580}%
\special{pa 1900 420}%
\special{dt 0.030}%
\special{sh 1}%
\special{pa 1900 420}%
\special{pa 1915 488}%
\special{pa 1926 466}%
\special{pa 1950 468}%
\special{pa 1900 420}%
\special{fp}%
%
\special{pn 13}%
\special{pa 1905 980}%
\special{pa 1950 800}%
\special{dt 0.030}%
\special{sh 1}%
\special{pa 1950 800}%
\special{pa 1914 860}%
\special{pa 1937 852}%
\special{pa 1953 870}%
\special{pa 1950 800}%
\special{fp}%
%
\special{pn 13}%
\special{pa 1950 800}%
\special{pa 1990 640}%
\special{dt 0.030}%
%
\special{pn 13}%
\special{pa 2085 740}%
\special{pa 2010 615}%
\special{dt 0.030}%
\special{sh 1}%
\special{pa 2010 615}%
\special{pa 2027 682}%
\special{pa 2037 661}%
\special{pa 2061 662}%
\special{pa 2010 615}%
\special{fp}%
\put(19.0000,-3.0000){\makebox(0,0){$f(x_2)$}}%
\put(19.0000,-11.0000){\makebox(0,0){$f(x_3)$}}%
\put(23.0000,-6.0000){\makebox(0,0)[lb]{$f(x_1)$}}%
%
\special{pn 8}%
\special{pa 2300 600}%
\special{pa 2120 735}%
\special{fp}%
\put(20.2500,-5.7500){\makebox(0,0)[lb]{$x_4$}}%
%
\special{pn 13}%
\special{pa 0 900}%
\special{pa 300 900}%
\special{fp}%
\special{sh 1}%
\special{pa 300 900}%
\special{pa 233 880}%
\special{pa 247 900}%
\special{pa 233 920}%
\special{pa 300 900}%
\special{fp}%
%
\special{pn 13}%
\special{pa 500 900}%
\special{pa 800 900}%
\special{fp}%
\special{sh 1}%
\special{pa 800 900}%
\special{pa 733 880}%
\special{pa 747 900}%
\special{pa 733 920}%
\special{pa 800 900}%
\special{fp}%
%
\special{pn 13}%
\special{pa 1000 900}%
\special{pa 1300 900}%
\special{fp}%
\special{sh 1}%
\special{pa 1300 900}%
\special{pa 1233 880}%
\special{pa 1247 900}%
\special{pa 1233 920}%
\special{pa 1300 900}%
\special{fp}%
%
\special{pn 13}%
\special{pa 150 900}%
\special{pa 400 650}%
\special{fp}%
\special{sh 1}%
\special{pa 400 650}%
\special{pa 339 683}%
\special{pa 362 688}%
\special{pa 367 711}%
\special{pa 400 650}%
\special{fp}%
%
\special{pn 13}%
\special{pa 400 650}%
\special{pa 650 400}%
\special{fp}%
%
\special{pn 13}%
\special{pa 650 400}%
\special{pa 650 650}%
\special{fp}%
\special{sh 1}%
\special{pa 650 650}%
\special{pa 670 583}%
\special{pa 650 597}%
\special{pa 630 583}%
\special{pa 650 650}%
\special{fp}%
%
\special{pn 13}%
\special{pa 650 650}%
\special{pa 650 900}%
\special{fp}%
%
\special{pn 13}%
\special{pa 1150 900}%
\special{pa 900 650}%
\special{fp}%
\special{sh 1}%
\special{pa 900 650}%
\special{pa 933 711}%
\special{pa 938 688}%
\special{pa 961 683}%
\special{pa 900 650}%
\special{fp}%
%
\special{pn 13}%
\special{pa 900 650}%
\special{pa 650 400}%
\special{fp}%
\put(6.5000,-11.0000){\makebox(0,0){$\Y$}}%
\end{picture}}%
\caption{}
\label{fig:localized_Y}
\end{figure}
Thus $\supp(\omega_{\Y})$ is contained in the subspace of $(x_1,x_2,x_3;x_4)\in C_Y(f)$ with $f_i(x_i)$ ($i=1,2,3$) ``near a triple point of the diagram $D$'', and is empty if $\epsilon$ is sufficiently small, since $D$ has no triple point.
This implies
\begin{equation}\label{eq:I_Y_doublepoint_limit}
 \lim_{\epsilon\to 0}I_{\Y}(f)=0.
\end{equation}
\eqref{eq:I_X_1_doublepoint_limit}, \eqref{eq:I_X_23_doublepoint_limit} and \eqref{eq:I_Y_doublepoint_limit} complete the proof of \eqref{eq:Polyak_formula_for_triple_linking}.

Almost the same argument proves \eqref{eq:Polyak_formula_for_Casson}, just replacing $\X_k$ and $\Y$ respectively with $X_k$ and $Y$.

\section{Polyak type formula for Milnor triple linking number from link diagrams with multiple crossings}\label{s:CSI_multiple_crossings}
Below we suppose that diagrams of long knots and links may have transverse multiple (triple or more) crossings.
The pairs of double crossings that do not form triple or quadruple points contribute to the Casson invariant and the Milnor triple linking number by the product of the signs of these crossings, as we see in \S\ref{ss:double_contribution_X}.
In this case triple or more crossings may also contribute to these invaraints, because $\supp(\omega_*)$ contains configurations as shown in Figures~\ref{fig:triple_contribution_X}, \ref{fig:triple_contribution_Y}, where all the four points get togather near a triple point.
We compute these contributions in the limit $\epsilon\to 0$.
\begin{figure}
\centering
\input{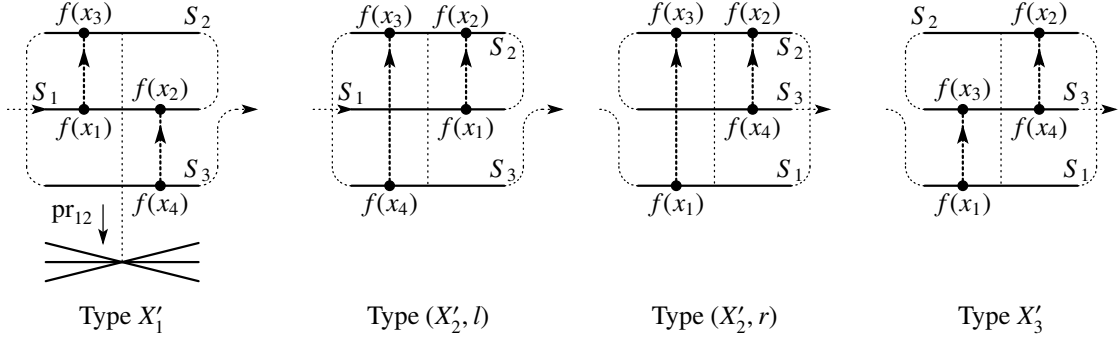}
\caption{Configurations that can contribute to $I_X$ and $I_{\X_k}$}
\label{fig:triple_contribution_X}
\end{figure}
\begin{figure}
\centering
\input{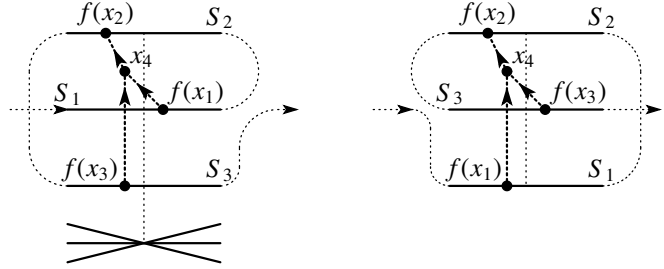}
\caption{Configurations that can contribute to $I_Y$ and $I_{\Y}$}
\label{fig:triple_contribution_Y}
\end{figure}

\subsection{A ``model'' for triple points}\label{ss:model_triple}
Consider three disjoint small straight segments $S_1$, $S_2$ and $S_3$ in $\R^3$ that are parallel to $N^{\perp}$ and whose projections onto $N^{\perp}$ make a transeverse triple point at their midpoints.
We regard these segments as
\begin{itemize}
\item
	subarcs of a long knot $f$ such that, if we follow $f$ along its orientation, we pass through these segments in order of $S_1$, $S_2$ and then $S_3$, or
\item
	subarcs of a long link $f$ with $S_k$ on the $k$-th component $f_k$ ($k=1,2,3$) of $f$.
\end{itemize}
To compute the contribution of triple points to $I_X$ and $I_{\X_k}$,
we will consider configurations of four points on the segments $S_k$ and the maps $h_{ij}$ defined similaly to \eqref{eq:h_X} and \eqref{eq:maps_for_X1},
then compute the integral defined similarly to \eqref{eq:I_X_knot} and \eqref{eq:I_X_1}.
The contribution of triple points to $I_Y$ and $I_{\Y}$ will be computed in similar way.

We only need to consider the case where $S_2$ is ``above'' both $S_1$ and $S_3$, since otherwise $\supp(\omega_*)=\varnothing$ ($*=X,Y$, see Figures~\ref{fig:triple_contribution_X}, \ref{fig:triple_contribution_Y}).
In such cases the triple point corresponds to one of the Gauss diagrams in Figure~\ref{fig:all_triple_arrow}.
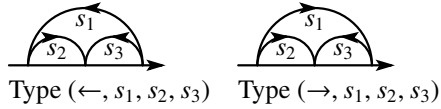
\begin{figure}
\centering
{\unitlength 0.1in%
\begin{picture}(20.0000,3.5500)(5.0000,-6.5500)%
%
\special{pn 13}%
\special{pa 500 600}%
\special{pa 1300 600}%
\special{fp}%
\special{sh 1}%
\special{pa 1300 600}%
\special{pa 1233 580}%
\special{pa 1247 600}%
\special{pa 1233 620}%
\special{pa 1300 600}%
\special{fp}%
%
\special{pn 13}%
\special{ar 900 600 300 300 4.7123890 6.2831853}%
%
\special{pn 13}%
\special{pa 915 300}%
\special{pa 900 300}%
\special{fp}%
\special{sh 1}%
\special{pa 900 300}%
\special{pa 967 320}%
\special{pa 953 300}%
\special{pa 967 280}%
\special{pa 900 300}%
\special{fp}%
%
\special{pn 13}%
\special{ar 900 600 300 300 3.1415927 4.7123890}%
%
\special{pn 13}%
\special{ar 750 600 150 150 4.7123890 6.2831853}%
%
\special{pn 13}%
\special{ar 750 600 150 150 3.1415927 4.7123890}%
%
\special{pn 13}%
\special{pa 735 451}%
\special{pa 750 450}%
\special{fp}%
\special{sh 1}%
\special{pa 750 450}%
\special{pa 682 434}%
\special{pa 697 454}%
\special{pa 685 474}%
\special{pa 750 450}%
\special{fp}%
%
\special{pn 13}%
\special{ar 1050 600 150 150 4.7123890 6.2831853}%
%
\special{pn 13}%
\special{pa 1065 451}%
\special{pa 1050 450}%
\special{fp}%
\special{sh 1}%
\special{pa 1050 450}%
\special{pa 1115 474}%
\special{pa 1103 454}%
\special{pa 1118 434}%
\special{pa 1050 450}%
\special{fp}%
%
\special{pn 13}%
\special{ar 1050 600 150 150 3.1415927 4.7123890}%
\put(9.0000,-3.7500){\makebox(0,0){$s_1$}}%
\put(7.5000,-5.2500){\makebox(0,0){$s_2$}}%
\put(10.5000,-5.2500){\makebox(0,0){$s_3$}}%
%
\special{pn 13}%
\special{pa 1700 600}%
\special{pa 2500 600}%
\special{fp}%
\special{sh 1}%
\special{pa 2500 600}%
\special{pa 2433 580}%
\special{pa 2447 600}%
\special{pa 2433 620}%
\special{pa 2500 600}%
\special{fp}%
\put(5.0000,-8.0000){\makebox(0,0)[lb]{Type $(\leftarrow,s_1,s_2,s_3)$}}%
%
\special{pn 13}%
\special{ar 2100 600 300 300 4.7123890 6.2831853}%
%
\special{pn 13}%
\special{ar 2100 600 300 300 3.1415927 4.7123890}%
%
\special{pn 13}%
\special{pa 2085 300}%
\special{pa 2100 300}%
\special{fp}%
\special{sh 1}%
\special{pa 2100 300}%
\special{pa 2033 280}%
\special{pa 2047 300}%
\special{pa 2033 320}%
\special{pa 2100 300}%
\special{fp}%
%
\special{pn 13}%
\special{ar 1950 600 150 150 3.1415927 4.7123890}%
%
\special{pn 13}%
\special{pa 1935 451}%
\special{pa 1950 450}%
\special{fp}%
\special{sh 1}%
\special{pa 1950 450}%
\special{pa 1882 434}%
\special{pa 1897 454}%
\special{pa 1885 474}%
\special{pa 1950 450}%
\special{fp}%
%
\special{pn 13}%
\special{ar 1950 600 150 150 4.7123890 6.2831853}%
%
\special{pn 13}%
\special{ar 2250 600 150 150 3.1415927 4.7123890}%
%
\special{pn 13}%
\special{ar 2250 600 150 150 4.7123890 6.2831853}%
%
\special{pn 13}%
\special{pa 2265 451}%
\special{pa 2250 450}%
\special{fp}%
\special{sh 1}%
\special{pa 2250 450}%
\special{pa 2315 474}%
\special{pa 2303 454}%
\special{pa 2318 434}%
\special{pa 2250 450}%
\special{fp}%
\put(21.0000,-3.7500){\makebox(0,0){$s_1$}}%
\put(19.5000,-5.2500){\makebox(0,0){$s_2$}}%
\put(22.5000,-5.2500){\makebox(0,0){$s_3$}}%
\put(17.0000,-8.0000){\makebox(0,0)[lb]{Type $(\rightarrow,s_1,s_2,s_3)$}}%
\end{picture}}%
\caption{All possible Gauss diagrams for triple points that can contribute to $I_X$ and $I_Y$ ($s_1,s_2,s_3\in\{+,-\}$)}
\label{fig:all_triple_arrow}
\end{figure}
We say the triple point is \emph{of type} $(\leftarrow,s_1,s_2,s_3)$ (resp.~$(\rightarrow,s_1,s_2,s_3)$) if it corresponds to the left (resp.~right) Gauss diagram in Figure~\ref{fig:all_triple_arrow}.
The arrows ``$\leftarrow$'' and ``$\rightarrow$'' indicate the direction of the oriented chord connecting the right-most and the left-most points of the Gauss diagrams.

\subsection{Contribution of triple points to $I_Y$}\label{ss:triple_contribution_Y}
We first consider the case where the triple point is of type $(\leftarrow,s_1,s_2,s_3)$.
We parametrize $S_k$ ($k=1,2,3$) by affine maps $f_k\colon(-\delta,\delta)\to S_k$ so that the midpoint $f_k(0)$ corresponds to the triple point and the orientations are compatible with the signs $s_k$.
Define the maps $h_{i4}\colon(-\delta,\delta)^{\times 3}\times\R^3\to S^2$ and $h_{42}\colon(-\delta,\delta)^{\times 3}\times\R^3\to S^2$ as indicated by $\Y$ (Figure~\ref{fig:Y_link}).
If we put $h_Y\coloneqq h_{14}\times h_{42}\times h_{34}\colon(-\delta,\delta)^{\times 3}\times\R^3$,
then the integral
\begin{equation}\label{eq:contribution_of_triple_points_to_I_Y}
 I^{\leftarrow,s_1,s_2,s_3}_Y\coloneqq\int_{(-\delta,\delta)^{\times 3}\times\R^3}h^*_Y(\eta\times\eta\times\eta)
\end{equation}
is the contribution of a triple point of type $(\leftarrow,s_1,s_2,s_3)$ to $I_Y$.

\begin{lem}
$I^{\leftarrow,s_1,s_2,s_3}_Y=0$ for sufficiently small $\epsilon$.
\end{lem}
\begin{proof}
First note that if we replace
\begin{itemize}
\item
	the parametrization $f_2$ of $S_2$ with $g_2(t)\coloneqq f_2(-t)$, or
\item
	the parametrizations $f_1,f_3$ of $S_1,S_3$ respectively with $h_1(t)\coloneqq f_1(-t)$ and $h_3(t)\coloneqq f_3(-t)$,
\end{itemize}
then we obtain models of triple point of type $(\leftarrow,s_1,-s_2,-s_3)$, here $-s_k=\mp$ if $s_k=\pm$ (see Figure~\ref{fig:triple_Y}).
\begin{figure}
\centering
\input{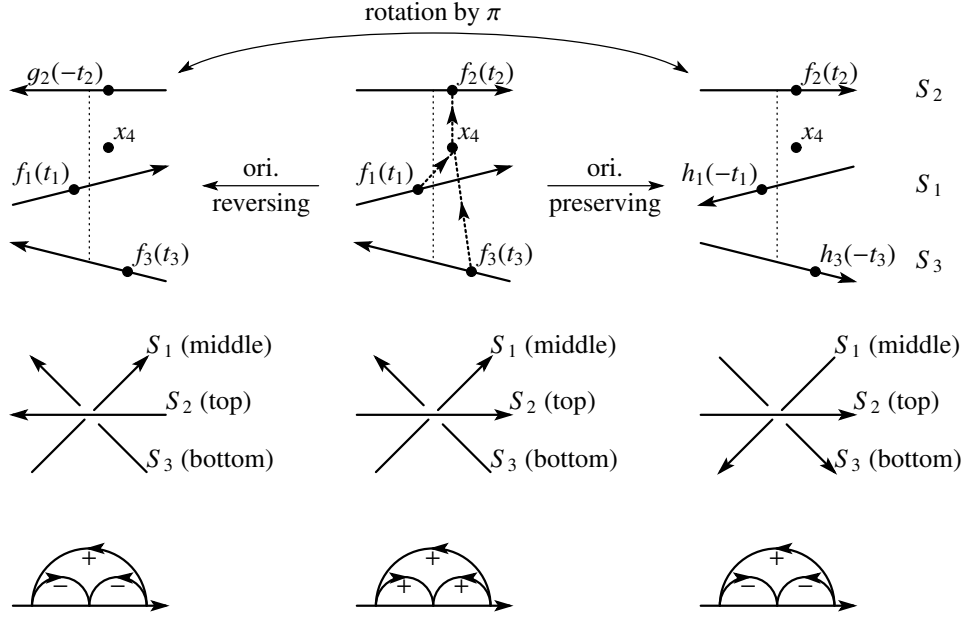}
\caption{The case that the triple point is of type $(\leftarrow,+,+,+)$}
\label{fig:triple_Y}
\end{figure}
The resulting models differ from each other by the $\pi$-rotation around $N$, but since we are assuming that $\eta$ is $\O(2)$-invariant, these models should give the same value $I^{\leftarrow,s_1,-s_2,-s_3}_Y$.

Define two diffeomorphisms $F,G\colon(-\delta,\delta)^{\times 3}\times\R^3\to(-\delta,\delta)^{\times 3}\times\R^3$ respectively by
\begin{equation}
 F(t_1,t_2,t_3;x_4)\coloneqq(t_1,-t_2,t_3;x_4)
 \quad\text{and}\quad
 G(t_1,t_2,t_3;x_4)\coloneqq(-t_1,t_2,-t_3;x_4).
\end{equation}
These diffeomorphisms make the following diagram commutative for both $F,G$.
\begin{equation}
\begin{split}
 \xymatrix{
  (-\delta,\delta)^{\times 3}\times\R^3\ar[r]^-{h_Y}\ar[d]_-{\cong}^-{\star} & (S^2)^{\times 3}\\
  (-\delta,\delta)^{\times 3}\times\R^3\ar[ru]_-{h_{Y,\star}} &
 }
\end{split}
\quad
(\star=F,G)
\end{equation}
Here $h_{Y,\star}$ is $h_Y$ defined with $f_2$ (resp.~$f_1$, $f_3$) replaced by $g_2$ (resp.~$h_1$, $h_3$) if $\star=F$ (resp.~$\star=G$).
Since $F$ reverses the orientation while $G$ preserves,
we have
\begin{equation}
 I^{\leftarrow,s_1,s_2,s_3}_Y
 =-\int_{(-\delta,\delta)^{\times 3}\times\R^3}h_{Y,F}^*(\eta\times\eta\times\eta)
 =+\int_{(-\delta,\delta)^{\times 3}\times\R^3}h_{Y,G}^*(\eta\times\eta\times\eta),
\end{equation}
and
\begin{equation}
 I^{\leftarrow,s_1,-s_2,-s_3}_Y=\int_{(-\delta,\delta)^{\times 3}\times\R^3}h_{Y,\star}^*(\eta\times\eta\times\eta)
 \quad\text{for both }
 \star=F,G.
\end{equation}
Thus, we have $I_Y^{\leftarrow,s_1,s_2,s_3}=I_Y^{\leftarrow,s_1,-s_2,-s_3}=-I_Y^{\leftarrow,s_1,-s_2,-s_3}$.
\end{proof}

The same holds for $I_Y^{\to,s_1,s_2,s_3}$.
Thus we have proved the following.

\begin{prop}\label{prop:triple_contribution_Y}
$I_Y^{\leftarrow,s_1,s_2,s_3}=I_Y^{\rightarrow,s_1,s_2,s_3}=0$ for any $s_1,s_2,s_3\in\{+,-\}$ and sufficiently small $\epsilon>0$.
\end{prop}

As for 3-component long links,
we need to investigate the contribution $I_{\Y}^{\leftarrow,s_1,s_2,s_3}$ and $I_{\Y}^{\to,s_1,s_2,s_3}$ of triple points that involve all the three components.
These contributions are similarly defined to \eqref{eq:contribution_of_triple_points_to_I_Y} with $S_k$ ($k=1,2,3$) replaced by segments on the $k$-th component of the link that form a triple point.
The same argument as above shows that $I_{\Y}^{\leftarrow,s_1,s_2,s_3}=I_{\Y}^{\rightarrow,s_1,s_2,s_3}=0$ for sufficiently small $\epsilon>0$.

\subsection{Contribution of triple points to $I_X$}
Consider the model for a triple point formed by the segments $S_1$, $S_2$ and $S_3$, as in \S\ref{ss:triple_contribution_Y}.
To compute the contribution of the configurations of types $X'_1$,
suppose that the triple point is of type $(\leftarrow,s_1,s_2,s_3)$ and each $S_k$ is parametrized by an affine map $f_k\colon(-\delta,\delta)\to S_k$.
Define the maps $h_{13},h_{42}\colon(-\delta,\delta)^{\times 4}\to S^2$ as indicated by Figure~\ref{fig:triple_contribution_X} (the left-most),
and put $h_{X'_1}=h_{13}\times h_{42}\colon(-\delta,\delta)^{\times 4}\to(S^2)^{\times 2}$.
Then the integral
\begin{equation}
 I_X^{X'_1,s_1,s_2,s_3}\coloneqq\int_{(-\delta,\delta)^{\times 4}}h_{X'_1}^*(\eta\times\eta)
\end{equation}
is the contributions of configurations of type $X'_1$ to $I_X$.
In the same ways the contribution $I_X^{*,s_1,s_2,s_3}$ of configurations of type $*$ to $I_X$, $*=(X'_2.l),(X'_2,r),X'_3$, are defined.
For $*=(X'_2,r)$ and $X'_3$,
we need to assume that the triple point is of type $(\rightarrow,s_1,s_2,s_3)$.

\begin{lem}\label{lem:I_X'}
For all $s_1,s_2,s_3\in\{+,-\}$, we have
\begin{enumerate}[(1)]
\item
	$I_X^{X'_k,s_1,s_2,s_3}=-I_X^{X'_k,s_1,-s_2,-s_3}=I_X^{X'_k,-s_1,-s_2,-s_3}$ ($k=1,3$),\item
	$I_X^{(X'_2,*),s_1,s_2,s_3}=I_X^{(X'_2,*),s_1,-s_2,-s_3}=I_X^{(X'_2,*),-s_1,-s_2,-s_3}$ for both $*=l,r$,
\end{enumerate}
where $-s_k=\mp$ for $s_k=\pm$,
\end{lem}
\begin{proof}
\begin{enumerate}[(1)]
\item
	We first prove (1) for $X'_1$.
	Consider the reflection $\rho\colon\R^3\to\R^3$ with respect to the plane that contains $\R N$ and $S_2$ (see Figure~\ref{fig:triple_X'1_reflection}).
	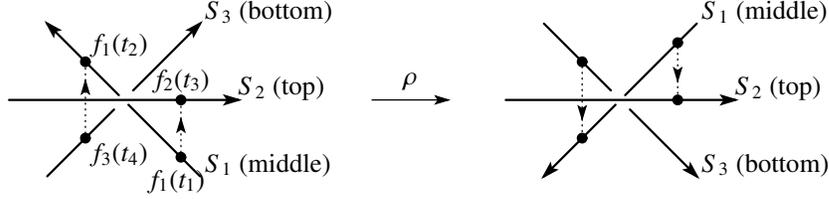
\begin{figure}
	\centering
{\unitlength 0.1in%
\begin{picture}(38.0000,9.4500)(2.0000,-11.0000)%
%
\special{pn 13}%
\special{pa 1200 1100}%
\special{pa 825 725}%
\special{fp}%
%
\special{pn 13}%
\special{pa 400 1100}%
\special{pa 750 750}%
\special{fp}%
\put(14.0000,-7.0000){\makebox(0,0)[lb]{$S_2$ (top)}}%
%
\special{pn 13}%
\special{pa 200 700}%
\special{pa 1400 700}%
\special{fp}%
\special{sh 1}%
\special{pa 1400 700}%
\special{pa 1333 680}%
\special{pa 1347 700}%
\special{pa 1333 720}%
\special{pa 1400 700}%
\special{fp}%
%
\special{pn 13}%
\special{pa 775 675}%
\special{pa 400 300}%
\special{fp}%
\special{sh 1}%
\special{pa 400 300}%
\special{pa 433 361}%
\special{pa 438 338}%
\special{pa 461 333}%
\special{pa 400 300}%
\special{fp}%
\put(12.2500,-11.0000){\makebox(0,0)[lb]{$S_1$ (middle)}}%
%
\special{pn 13}%
\special{pa 850 650}%
\special{pa 1200 300}%
\special{fp}%
\special{sh 1}%
\special{pa 1200 300}%
\special{pa 1139 333}%
\special{pa 1162 338}%
\special{pa 1167 361}%
\special{pa 1200 300}%
\special{fp}%
\put(12.2500,-3.0000){\makebox(0,0)[lb]{$S_3$ (bottom)}}%
\put(23.0000,-6.0000){\makebox(0,0){$\rho$}}%
%
\special{pn 8}%
\special{pa 2100 700}%
\special{pa 2500 700}%
\special{fp}%
\special{sh 1}%
\special{pa 2500 700}%
\special{pa 2433 680}%
\special{pa 2447 700}%
\special{pa 2433 720}%
\special{pa 2500 700}%
\special{fp}%
%
\special{pn 13}%
\special{pa 2800 700}%
\special{pa 4000 700}%
\special{fp}%
\special{sh 1}%
\special{pa 4000 700}%
\special{pa 3933 680}%
\special{pa 3947 700}%
\special{pa 3933 720}%
\special{pa 4000 700}%
\special{fp}%
\put(40.0000,-7.0000){\makebox(0,0)[lb]{$S_2$ (top)}}%
%
\special{pn 13}%
\special{pa 3800 300}%
\special{pa 3425 675}%
\special{fp}%
%
\special{pn 13}%
\special{pa 3375 725}%
\special{pa 3000 1100}%
\special{fp}%
\special{sh 1}%
\special{pa 3000 1100}%
\special{pa 3061 1067}%
\special{pa 3038 1062}%
\special{pa 3033 1039}%
\special{pa 3000 1100}%
\special{fp}%
\put(38.2500,-11.0000){\makebox(0,0)[lb]{$S_3$ (bottom)}}%
\put(38.2500,-3.0000){\makebox(0,0)[lb]{$S_1$ (middle)}}%
%
\special{pn 13}%
\special{pa 3000 300}%
\special{pa 3350 650}%
\special{fp}%
%
\special{pn 13}%
\special{pa 3450 750}%
\special{pa 3800 1100}%
\special{fp}%
\special{sh 1}%
\special{pa 3800 1100}%
\special{pa 3767 1039}%
\special{pa 3762 1062}%
\special{pa 3739 1067}%
\special{pa 3800 1100}%
\special{fp}%
%
\special{sh 1.000}%
\special{ia 1100 700 25 25 0.0000000 6.2831853}%
\special{pn 8}%
\special{ar 1100 700 25 25 0.0000000 6.2831853}%
%
\special{sh 1.000}%
\special{ia 1100 1000 25 25 0.0000000 6.2831853}%
\special{pn 8}%
\special{ar 1100 1000 25 25 0.0000000 6.2831853}%
%
\special{sh 1.000}%
\special{ia 600 500 25 25 0.0000000 6.2831853}%
\special{pn 8}%
\special{ar 600 500 25 25 0.0000000 6.2831853}%
%
\special{sh 1.000}%
\special{ia 600 900 25 25 0.0000000 6.2831853}%
\special{pn 8}%
\special{ar 600 900 25 25 0.0000000 6.2831853}%
\put(11.0000,-6.0000){\makebox(0,0){$f_2(t_3)$}}%
\put(10.7500,-11.2500){\makebox(0,0){$f_1(t_1)$}}%
%
\special{pn 8}%
\special{pa 1100 975}%
\special{pa 1100 800}%
\special{dt 0.030}%
\special{sh 1}%
\special{pa 1100 800}%
\special{pa 1080 867}%
\special{pa 1100 853}%
\special{pa 1120 867}%
\special{pa 1100 800}%
\special{fp}%
%
\special{pn 8}%
\special{pa 1100 800}%
\special{pa 1100 725}%
\special{dt 0.030}%
\put(6.2500,-4.7500){\makebox(0,0)[lb]{$f_1(t_2)$}}%
\put(6.2500,-9.2500){\makebox(0,0)[lt]{$f_3(t_4)$}}%
%
\special{pn 8}%
\special{pa 600 875}%
\special{pa 600 600}%
\special{dt 0.030}%
\special{sh 1}%
\special{pa 600 600}%
\special{pa 580 667}%
\special{pa 600 653}%
\special{pa 620 667}%
\special{pa 600 600}%
\special{fp}%
%
\special{pn 8}%
\special{pa 600 600}%
\special{pa 600 525}%
\special{dt 0.030}%
%
\special{sh 1.000}%
\special{ia 3700 700 25 25 0.0000000 6.2831853}%
\special{pn 8}%
\special{ar 3700 700 25 25 0.0000000 6.2831853}%
%
\special{sh 1.000}%
\special{ia 3700 400 25 25 0.0000000 6.2831853}%
\special{pn 8}%
\special{ar 3700 400 25 25 0.0000000 6.2831853}%
%
\special{sh 1.000}%
\special{ia 3200 500 25 25 0.0000000 6.2831853}%
\special{pn 8}%
\special{ar 3200 500 25 25 0.0000000 6.2831853}%
%
\special{sh 1.000}%
\special{ia 3200 900 25 25 0.0000000 6.2831853}%
\special{pn 8}%
\special{ar 3200 900 25 25 0.0000000 6.2831853}%
%
\special{pn 8}%
\special{pa 3700 425}%
\special{pa 3700 600}%
\special{dt 0.030}%
\special{sh 1}%
\special{pa 3700 600}%
\special{pa 3720 533}%
\special{pa 3700 547}%
\special{pa 3680 533}%
\special{pa 3700 600}%
\special{fp}%
%
\special{pn 8}%
\special{pa 3200 525}%
\special{pa 3200 800}%
\special{dt 0.030}%
\special{sh 1}%
\special{pa 3200 800}%
\special{pa 3220 733}%
\special{pa 3200 747}%
\special{pa 3180 733}%
\special{pa 3200 800}%
\special{fp}%
%
\special{pn 8}%
\special{pa 3200 800}%
\special{pa 3200 875}%
\special{dt 0.030}%
%
\special{pn 8}%
\special{pa 3700 600}%
\special{pa 3700 675}%
\special{dt 0.030}%
\end{picture}}%
	\caption{The case for $X'_1$; the triple point changes its type from $(\leftarrow,-,+,+)$ to $(\leftarrow,+,-,-)$}
	\label{fig:triple_X'1_reflection}
	\end{figure}
	Then $\rho\circ f_k$ ($k=1,2,3$) form the model for the triple point of type $(\leftarrow,-\epsilon_1,-\epsilon_2,-\epsilon_3)$.
	Define $\varphi\colon(S^2)^{\times 2}\to(S^2)^{\times 2}$ by $\varphi\coloneqq\rho\times\rho$.
	Then we have a commutative diagram
	\begin{equation}
	\begin{split}
	 \xymatrix{
	 (-\delta,\delta)^{\times 4}\ar[r]^-{h_{X'_1}}\ar[rd]_-{\widehat{h}_{X'_1}} & (S^2)^{\times 2}\ar[d]^-{\varphi} \\
	  & (S^2)^{\times 2}
	 }
	\end{split}
	\end{equation}
	here $\widehat{h}_{X'_1}$ is $h_{X'_1}$ defined with $f_k$ replaced by $\rho\circ f_k$.
	This proves $I_X^{X'_1,s_1,s_2,s_3}=I_X^{X'_1,-s_1,-s_2,-s_3}$, since $\varphi$ preserves the orientation and we assume that $\eta$ is $\O(2)$-invariant.

	Let $g_k\colon(-\delta,\delta)\to S_k$ ($k=1,3$) be the parametrizations of $S_k$ with opposite orientation, namely $g_k(t)\coloneqq f_k(-t)$.
	If $f_1,f_2,f_3$ form a triple point of type $(\leftarrow,s_1,s_2,s_3)$,
	then $g_1,f_2,g_3$ form a triple point of type $(\leftarrow,s_1,-s_2,-s_3)$ (see Figure~\ref{fig:triple_X'1_reverse}).
	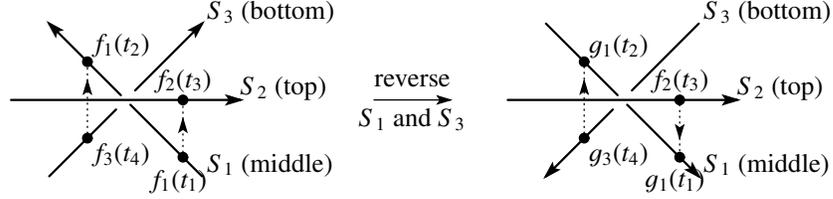
\begin{figure}
	\centering
{\unitlength 0.1in%
\begin{picture}(38.0000,9.4500)(2.0000,-11.0000)%
%
\special{pn 13}%
\special{pa 1200 1100}%
\special{pa 825 725}%
\special{fp}%
%
\special{pn 13}%
\special{pa 400 1100}%
\special{pa 750 750}%
\special{fp}%
\put(14.0000,-7.0000){\makebox(0,0)[lb]{$S_2$ (top)}}%
%
\special{pn 13}%
\special{pa 200 700}%
\special{pa 1400 700}%
\special{fp}%
\special{sh 1}%
\special{pa 1400 700}%
\special{pa 1333 680}%
\special{pa 1347 700}%
\special{pa 1333 720}%
\special{pa 1400 700}%
\special{fp}%
%
\special{pn 13}%
\special{pa 775 675}%
\special{pa 400 300}%
\special{fp}%
\special{sh 1}%
\special{pa 400 300}%
\special{pa 433 361}%
\special{pa 438 338}%
\special{pa 461 333}%
\special{pa 400 300}%
\special{fp}%
\put(12.2500,-11.0000){\makebox(0,0)[lb]{$S_1$ (middle)}}%
%
\special{pn 13}%
\special{pa 850 650}%
\special{pa 1200 300}%
\special{fp}%
\special{sh 1}%
\special{pa 1200 300}%
\special{pa 1139 333}%
\special{pa 1162 338}%
\special{pa 1167 361}%
\special{pa 1200 300}%
\special{fp}%
\put(12.2500,-3.0000){\makebox(0,0)[lb]{$S_3$ (bottom)}}%
%
\special{pn 8}%
\special{pa 2100 700}%
\special{pa 2500 700}%
\special{fp}%
\special{sh 1}%
\special{pa 2500 700}%
\special{pa 2433 680}%
\special{pa 2447 700}%
\special{pa 2433 720}%
\special{pa 2500 700}%
\special{fp}%
%
\special{pn 13}%
\special{pa 2800 700}%
\special{pa 4000 700}%
\special{fp}%
\special{sh 1}%
\special{pa 4000 700}%
\special{pa 3933 680}%
\special{pa 3947 700}%
\special{pa 3933 720}%
\special{pa 4000 700}%
\special{fp}%
\put(40.0000,-7.0000){\makebox(0,0)[lb]{$S_2$ (top)}}%
%
\special{pn 13}%
\special{pa 3800 300}%
\special{pa 3450 650}%
\special{fp}%
%
\special{pn 13}%
\special{pa 3350 750}%
\special{pa 3000 1100}%
\special{fp}%
\special{sh 1}%
\special{pa 3000 1100}%
\special{pa 3061 1067}%
\special{pa 3038 1062}%
\special{pa 3033 1039}%
\special{pa 3000 1100}%
\special{fp}%
\put(38.2500,-11.0000){\makebox(0,0)[lb]{$S_1$ (middle)}}%
\put(38.2500,-3.0000){\makebox(0,0)[lb]{$S_3$ (bottom)}}%
%
\special{pn 13}%
\special{pa 3000 300}%
\special{pa 3375 675}%
\special{fp}%
%
\special{pn 13}%
\special{pa 3425 725}%
\special{pa 3800 1100}%
\special{fp}%
\special{sh 1}%
\special{pa 3800 1100}%
\special{pa 3767 1039}%
\special{pa 3762 1062}%
\special{pa 3739 1067}%
\special{pa 3800 1100}%
\special{fp}%
%
\special{sh 1.000}%
\special{ia 1100 700 25 25 0.0000000 6.2831853}%
\special{pn 8}%
\special{ar 1100 700 25 25 0.0000000 6.2831853}%
%
\special{sh 1.000}%
\special{ia 1100 1000 25 25 0.0000000 6.2831853}%
\special{pn 8}%
\special{ar 1100 1000 25 25 0.0000000 6.2831853}%
%
\special{sh 1.000}%
\special{ia 600 500 25 25 0.0000000 6.2831853}%
\special{pn 8}%
\special{ar 600 500 25 25 0.0000000 6.2831853}%
%
\special{sh 1.000}%
\special{ia 600 900 25 25 0.0000000 6.2831853}%
\special{pn 8}%
\special{ar 600 900 25 25 0.0000000 6.2831853}%
\put(11.0000,-6.0000){\makebox(0,0){$f_2(t_3)$}}%
\put(10.7500,-11.2500){\makebox(0,0){$f_1(t_1)$}}%
%
\special{pn 8}%
\special{pa 1100 975}%
\special{pa 1100 800}%
\special{dt 0.030}%
\special{sh 1}%
\special{pa 1100 800}%
\special{pa 1080 867}%
\special{pa 1100 853}%
\special{pa 1120 867}%
\special{pa 1100 800}%
\special{fp}%
%
\special{pn 8}%
\special{pa 1100 800}%
\special{pa 1100 725}%
\special{dt 0.030}%
\put(6.2500,-4.7500){\makebox(0,0)[lb]{$f_1(t_2)$}}%
\put(6.2500,-9.2500){\makebox(0,0)[lt]{$f_3(t_4)$}}%
%
\special{pn 8}%
\special{pa 600 875}%
\special{pa 600 600}%
\special{dt 0.030}%
\special{sh 1}%
\special{pa 600 600}%
\special{pa 580 667}%
\special{pa 600 653}%
\special{pa 620 667}%
\special{pa 600 600}%
\special{fp}%
%
\special{pn 8}%
\special{pa 600 600}%
\special{pa 600 525}%
\special{dt 0.030}%
%
\special{sh 1.000}%
\special{ia 3700 700 25 25 0.0000000 6.2831853}%
\special{pn 8}%
\special{ar 3700 700 25 25 0.0000000 6.2831853}%
%
\special{sh 1.000}%
\special{ia 3700 1000 25 25 0.0000000 6.2831853}%
\special{pn 8}%
\special{ar 3700 1000 25 25 0.0000000 6.2831853}%
%
\special{sh 1.000}%
\special{ia 3200 500 25 25 0.0000000 6.2831853}%
\special{pn 8}%
\special{ar 3200 500 25 25 0.0000000 6.2831853}%
%
\special{sh 1.000}%
\special{ia 3200 900 25 25 0.0000000 6.2831853}%
\special{pn 8}%
\special{ar 3200 900 25 25 0.0000000 6.2831853}%
%
\special{pn 8}%
\special{pa 3700 725}%
\special{pa 3700 900}%
\special{dt 0.030}%
\special{sh 1}%
\special{pa 3700 900}%
\special{pa 3720 833}%
\special{pa 3700 847}%
\special{pa 3680 833}%
\special{pa 3700 900}%
\special{fp}%
%
\special{pn 8}%
\special{pa 3200 875}%
\special{pa 3200 600}%
\special{dt 0.030}%
\special{sh 1}%
\special{pa 3200 600}%
\special{pa 3180 667}%
\special{pa 3200 653}%
\special{pa 3220 667}%
\special{pa 3200 600}%
\special{fp}%
%
\special{pn 8}%
\special{pa 3700 900}%
\special{pa 3700 975}%
\special{dt 0.030}%
\put(23.0000,-6.0000){\makebox(0,0){reverse}}%
\put(23.0000,-8.0000){\makebox(0,0){$S_1$ and $S_3$}}%
\put(32.2500,-4.8000){\makebox(0,0)[lb]{$g_1(t_2)$}}%
\put(36.7500,-11.2500){\makebox(0,0){$g_1(t_1)$}}%
\put(32.2500,-9.2500){\makebox(0,0)[lt]{$g_3(t_4)$}}%
\put(37.0000,-6.0000){\makebox(0,0){$f_2(t_3)$}}%
%
\special{pn 8}%
\special{pa 3200 525}%
\special{pa 3200 600}%
\special{dt 0.030}%
\end{picture}}%
	\caption{The case for $X'_1$; the triple point changes its type from $(\leftarrow,-,+,+)$ to $(\leftarrow,-,-,-)$}
	\label{fig:triple_X'1_reverse}
	\end{figure}
	Define $F\colon(-\delta,\delta)^{\times 4}\to(-\delta,\delta)^{\times 4}$ by $F(t_1,\dots,t_4)\coloneqq(-t_1,-t_2,t_3,-t_4)$,
	then we have a commutative diagram
	\begin{equation}
	\begin{split}
	 \xymatrix{
	  (-\delta,\delta)^{\times 4}\ar[r]^-{h_{X'_1}}\ar[d]_-F & (S^2)^{\times 2} \\
	  (-\delta,\delta)^{\times 4}\ar[ru]_-{\underline{h}_{X'_1}} &
	 }
	\end{split}
	\end{equation}
	here $\underline{h}^{X'_1}$ is $h^{X'_1}$ defined with $f_1,f_3$ replaced respectively by $g_1,g_3$.
	This proves $I_X^{X'_1,s_1,s_2,s_3}=-I_X^{X'_1,s_1,-s_2,-s_3}$, since $F$ reverses the orientation.

	The proof of (1) for $X'_3$ is the same, replacing $F$ with $G(t_1,\dots,t_4)=(-t_1,t_2,-t_3,-t_4)$.
\item
	The proof of $I_X^{(X'_2,*),s_1,s_2,s_3}=I_X^{(X'_2,*),-s_1,-s_2,-s_3}$ is the same as (1).
	That of $I_X^{(X'_2,*),s_1,s_2,s_3}=I_X^{(X'_2,*),s_1,-s_2,-s_3}$ is also similar to (1), but we need to replace $F$ and $G$ with $H(t_1,\dots,t_4)=(-t_1,t_2,t_3,-t_4)$, that preserves orientations.\qedhere
\end{enumerate}
\end{proof}

Consider the long knot $f$ shown in Figure~\ref{fig:knot_1}.
\begin{figure}
\centering
\input{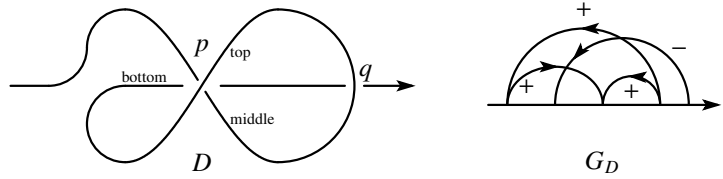}
\caption{The long knot $f$; the triple point is of type $(\leftarrow,+,+,+)$ and $\pair{X}{G_D}=1\cdot(-1)=-1$}
\label{fig:knot_1}
\end{figure}
Clearly $f$ is isotopic to the unknot, and hence $c(f)=0$.
$G_D$ is as in Figure~\ref{fig:knot_1}.
For sufficiently small $\epsilon>0$,
we have $I_X(f)-I_Y(f)=I_X(f)$ by Proposition~\ref{prop:triple_contribution_Y}.
There exists one pair of double crossings that contributes to $I_X(f)$ as $\pair{X}{G_D}=-1$;
one double crossing is formed by the bottom and the top segment of $p$,
and the other is $q$.
In addition,
because $p$ is a triple point of type $(\leftarrow,+,+,+)$,
the contributions of types $X'_1$ and $(X'_2,l)$ occur.
Thus we see that
\begin{equation}\label{eq:knot_1}
 0=c(f)=\pair{X}{G_D}+I^{X'_1,+,+,+}_X+I^{(X'_2,l),+,+,+}_X=-1+I^{X'_1,+,+,+}_X+I^{(X'_2,l),+,+,+}_X.
\end{equation}
Consider the long knot $g$ shown in Figure~\ref{fig:knot_2}.
\begin{figure}
\centering
{\unitlength 0.1in%
\begin{picture}(36.9200,9.7200)(9.0800,-30.2700)%
\put(20.0000,-25.0000){\makebox(0,0)[lb]{{\tiny middle}}}%
\put(19.5000,-22.0000){\makebox(0,0)[lb]{{\tiny top}}}%
\put(13.0000,-23.5000){\makebox(0,0){{\tiny bottom}}}%
%
\special{pn 13}%
\special{pa 3400 2800}%
\special{pa 4600 2800}%
\special{fp}%
\special{sh 1}%
\special{pa 4600 2800}%
\special{pa 4533 2780}%
\special{pa 4547 2800}%
\special{pa 4533 2820}%
\special{pa 4600 2800}%
\special{fp}%
%
\special{pn 13}%
\special{ar 3900 2800 400 400 4.7123890 6.2831853}%
%
\special{pn 13}%
\special{pa 3915 2400}%
\special{pa 3900 2400}%
\special{fp}%
\special{sh 1}%
\special{pa 3900 2400}%
\special{pa 3967 2420}%
\special{pa 3953 2400}%
\special{pa 3967 2380}%
\special{pa 3900 2400}%
\special{fp}%
%
\special{pn 13}%
\special{ar 4050 2800 250 200 4.1719695 6.2831853}%
%
\special{pn 13}%
\special{pa 3936 2622}%
\special{pa 3921 2629}%
\special{fp}%
\special{sh 1}%
\special{pa 3921 2629}%
\special{pa 3990 2619}%
\special{pa 3969 2606}%
\special{pa 3973 2583}%
\special{pa 3921 2629}%
\special{fp}%
%
\special{pn 13}%
\special{ar 3650 2800 150 150 3.1415927 4.7123890}%
%
\special{pn 13}%
\special{pa 3635 2651}%
\special{pa 3650 2650}%
\special{fp}%
\special{sh 1}%
\special{pa 3650 2650}%
\special{pa 3582 2634}%
\special{pa 3597 2654}%
\special{pa 3585 2674}%
\special{pa 3650 2650}%
\special{fp}%
%
\special{pn 13}%
\special{ar 4050 2800 250 200 3.1415927 4.1719695}%
%
\special{pn 13}%
\special{ar 3900 2800 400 400 3.1415927 4.7123890}%
%
\special{pn 13}%
\special{ar 3650 2800 150 150 4.7123890 6.2831853}%
\put(39.0000,-23.0000){\makebox(0,0){$+$}}%
\put(36.5000,-27.0000){\makebox(0,0){$-$}}%
\put(42.0000,-27.0000){\makebox(0,0){$-$}}%
\put(44.0000,-25.0000){\makebox(0,0){$+$}}%
\put(19.0000,-31.0000){\makebox(0,0){$D$}}%
\put(40.0000,-31.0000){\makebox(0,0){$G_D$}}%
%
\special{pn 13}%
\special{ar 1800 2600 200 200 6.2831853 3.1415927}%
%
\special{pn 13}%
\special{ar 1600 2600 400 400 5.3558901 6.2831853}%
%
\special{pn 13}%
\special{ar 2000 2600 400 400 3.1415927 4.7123890}%
%
\special{pn 13}%
\special{ar 1600 2600 400 400 4.7123890 5.1760366}%
%
\special{pn 13}%
\special{ar 2000 2600 400 400 4.7123890 1.5707963}%
%
\special{pn 13}%
\special{pa 1900 2250}%
\special{pa 2120 2250}%
\special{fp}%
%
\special{pn 13}%
\special{pa 2250 2250}%
\special{pa 2500 2250}%
\special{fp}%
\special{sh 1}%
\special{pa 2500 2250}%
\special{pa 2433 2230}%
\special{pa 2447 2250}%
\special{pa 2433 2270}%
\special{pa 2500 2250}%
\special{fp}%
%
\special{pn 13}%
\special{pa 1200 2200}%
\special{pa 1600 2200}%
\special{fp}%
%
\special{pn 13}%
\special{ar 2000 2600 600 400 1.5707963 3.1415927}%
%
\special{pn 13}%
\special{ar 1800 2600 400 350 3.1415927 4.4287068}%
%
\special{pn 13}%
\special{ar 4250 2800 250 250 5.4977871 6.2831853}%
%
\special{pn 13}%
\special{pa 4437 2634}%
\special{pa 4427 2623}%
\special{fp}%
\special{sh 1}%
\special{pa 4427 2623}%
\special{pa 4457 2686}%
\special{pa 4463 2662}%
\special{pa 4487 2659}%
\special{pa 4427 2623}%
\special{fp}%
%
\special{pn 13}%
\special{ar 4250 2800 250 250 3.1415927 5.4977871}%
\end{picture}}%
\caption{The long knot $g$; the triple point is of type $(\leftarrow,+,-,-)$ and $\pair{X}{G_D}=0$}
\label{fig:knot_2}
\end{figure}
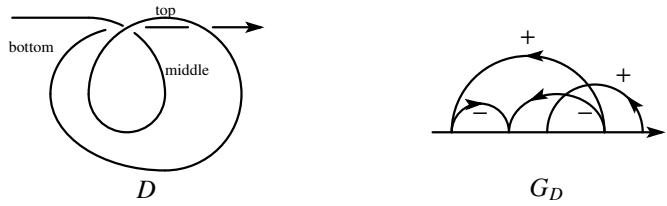
Clearly $g$ is isotopic to the unknot, and hence $c(g)=0$.
$G_D$ is as in Figure~\ref{fig:knot_2}.
In this case the triple point is of type $(\leftarrow,+,-,-)$,
and similar to the above,
we have
\begin{equation}\label{eq:knot_2}
 0=c(g)=\pair{X}{G_D}+I^{X'_1,+,-,-}_X+I^{(X'_2,l),+,-,-}_X=I^{X'_1,+,-,-}_X+I^{(X'_2,l),+,-,-}_X.
\end{equation}
By \eqref{eq:knot_1}, 
\eqref{eq:knot_2} and Lemma~\ref{lem:I_X'},
we have
\begin{equation}\label{eq:knot_3}
 I^{X'_1,+,+,+}_X=-I_X^{X'_1,+,-,-}=\frac{1}{2},
 \quad
 I^{(X'_2,l),+,+,+}_X=I_X^{(X'_2,l),+,-,-}=\frac{1}{2}.
\end{equation}
Lemma~\ref{lem:I_X'} and \eqref{eq:knot_3} show
\begin{equation}\label{eq:triple_X}
 I^{X'_1,s_1,s_2,s_3}_X=\frac{s_1s_2}{2},
 \quad
 I^{(X'_2,l),s_1,s_2,s_3}_X=I^{(X'_2,r),s_1,s_2,s_3}_X=\frac{s_2s_3}{2},
 \quad
 I^{X'_3,s_1,s_2,s_3}_X=\frac{s_3s_1}{2}.
\end{equation}

\subsection{Proof of \eqref{eq:generalized_Polyak_formula_for_Casson} and \eqref{eq:generalized_Polyak_formula_for_triple_linking}}
Let $f$ be a long knot with a diagram $D$.
We may suppose that the subarcs of $f$ near triple or more multiple crossings are parallel to $N^{\perp}$ like as $S_k$'s in \S\ref{ss:model_triple}.
By Proposition~\ref{prop:triple_contribution_Y},
we only need to compute $I_X(f)$ for sufficiently small $\epsilon>0$.

As we see in \S\ref{ss:double_contribution_X},
every pair $(\alpha,\beta)\in C^2_X(D)$ (Definition~\ref{def:pairing_for_knots}) contributes to $I_X$ by $\sgn(p)\sgn(q)$,
where $p,q$ are the double crossings of $D$ corresponding respectively to $\alpha$ and $\beta$.
Such contributions amount to $\pair{X}{G_D}$.

Figure~\ref{fig:triple_contribution_X} and \eqref{eq:triple_X} imply that every triple point of type $(\leftarrow,s_1,s_2,s_3)$ (resp.~$(\rightarrow,s_1,s_2,s_3)$) contributes to $I_X$ by $s_1s_2/2+s_2s_3/2$ (resp.~$s_3s_1/2+s_2s_3/2$).
Since $A_D$ never contains subdiagrams shown in Figure~\ref{fig:forbidden_diagrams},
the contributions of type $X'_1$ (resp.~$X'_3$) necessarily come from a triple point of type $(\leftarrow,s_1,s_2,s_3)$ (resp.~$(\rightarrow,s_1,s_2,s_3)$).
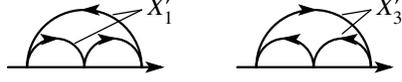
\begin{figure}
\centering
{\unitlength 0.1in%
\begin{picture}(20.0000,3.9400)(4.0000,-6.2100)%
%
\special{pn 13}%
\special{pa 400 600}%
\special{pa 1200 600}%
\special{fp}%
\special{sh 1}%
\special{pa 1200 600}%
\special{pa 1133 580}%
\special{pa 1147 600}%
\special{pa 1133 620}%
\special{pa 1200 600}%
\special{fp}%
%
\special{pn 13}%
\special{ar 650 600 150 150 4.7123890 6.2831853}%
%
\special{pn 13}%
\special{ar 950 600 150 150 3.1415927 4.7123890}%
%
\special{pn 13}%
\special{pa 935 451}%
\special{pa 950 450}%
\special{fp}%
\special{sh 1}%
\special{pa 950 450}%
\special{pa 882 434}%
\special{pa 897 454}%
\special{pa 885 474}%
\special{pa 950 450}%
\special{fp}%
%
\special{pn 13}%
\special{ar 800 600 300 300 4.7123890 6.2831853}%
%
\special{pn 13}%
\special{pa 815 300}%
\special{pa 800 300}%
\special{fp}%
\special{sh 1}%
\special{pa 800 300}%
\special{pa 867 320}%
\special{pa 853 300}%
\special{pa 867 280}%
\special{pa 800 300}%
\special{fp}%
%
\special{pn 13}%
\special{ar 800 600 300 300 3.1415927 4.7123890}%
%
\special{pn 13}%
\special{ar 650 600 150 150 3.1415927 4.7123890}%
%
\special{pn 13}%
\special{pa 635 451}%
\special{pa 650 450}%
\special{fp}%
\special{sh 1}%
\special{pa 650 450}%
\special{pa 582 434}%
\special{pa 597 454}%
\special{pa 585 474}%
\special{pa 650 450}%
\special{fp}%
%
\special{pn 13}%
\special{ar 950 600 150 150 4.7123890 6.2831853}%
%
\special{pn 13}%
\special{pa 1600 600}%
\special{pa 2400 600}%
\special{fp}%
\special{sh 1}%
\special{pa 2400 600}%
\special{pa 2333 580}%
\special{pa 2347 600}%
\special{pa 2333 620}%
\special{pa 2400 600}%
\special{fp}%
%
\special{pn 13}%
\special{ar 2000 600 300 300 3.1415927 4.7123890}%
%
\special{pn 13}%
\special{pa 1985 300}%
\special{pa 2000 300}%
\special{fp}%
\special{sh 1}%
\special{pa 2000 300}%
\special{pa 1933 280}%
\special{pa 1947 300}%
\special{pa 1933 320}%
\special{pa 2000 300}%
\special{fp}%
%
\special{pn 13}%
\special{ar 2000 600 300 300 4.7123890 6.2831853}%
%
\special{pn 13}%
\special{ar 1850 600 150 150 4.7123890 6.2831853}%
%
\special{pn 13}%
\special{pa 1865 451}%
\special{pa 1850 450}%
\special{fp}%
\special{sh 1}%
\special{pa 1850 450}%
\special{pa 1915 474}%
\special{pa 1903 454}%
\special{pa 1918 434}%
\special{pa 1850 450}%
\special{fp}%
%
\special{pn 13}%
\special{ar 2150 600 150 150 4.7123890 6.2831853}%
%
\special{pn 13}%
\special{pa 2165 451}%
\special{pa 2150 450}%
\special{fp}%
\special{sh 1}%
\special{pa 2150 450}%
\special{pa 2215 474}%
\special{pa 2203 454}%
\special{pa 2218 434}%
\special{pa 2150 450}%
\special{fp}%
%
\special{pn 13}%
\special{ar 1850 600 150 150 3.1415927 4.7123890}%
%
\special{pn 13}%
\special{ar 2150 600 150 150 3.1415927 4.7123890}%
\put(12.0000,-3.0000){\makebox(0,0){$X'_1$}}%
\put(24.0000,-3.0000){\makebox(0,0){$X'_3$}}%
%
\special{pn 8}%
\special{pa 1100 300}%
\special{pa 950 325}%
\special{fp}%
%
\special{pn 8}%
\special{pa 1100 300}%
\special{pa 750 475}%
\special{fp}%
%
\special{pn 8}%
\special{pa 2300 300}%
\special{pa 2150 325}%
\special{fp}%
%
\special{pn 8}%
\special{pa 2300 300}%
\special{pa 2150 425}%
\special{fp}%
\end{picture}}%
\caption{Subdiagrams that never appear in $A_D$}
\label{fig:forbidden_diagrams}
\end{figure}
Hence these contributions amount to
\begin{equation}
\begin{split}
 &\frac{1}{2}\sum_{s_1,s_2,s_3\in\{+,-\}}\sum_{\substack{\text{triple points of} \\ \text{type }(\leftarrow,s_1,s_2,s_3)}}(s_1s_2+s_2s_3)
 +\frac{1}{2}\sum_{s_1,s_2,s_3\in\{+,-\}}\sum_{\substack{\text{triple points of} \\ \text{type }(\rightarrow,s_1,s_2,s_3)}}(s_3s_1+s_2s_3)\\
 &=\frac{1}{2}\sum_{(\alpha,\beta)\in C^2_{X'_1}(D)}s_1s_2+\frac{1}{2}\sum_{(\alpha,\beta)\in C^2_{X'_2}(D)}s_2s_3+\frac{1}{2}\sum_{(\alpha,\beta)\in C^2_{X'_3}(D)}s_3s_1=\frac{1}{2}\left(\pair{X'_1}{G_D}+\pair{X'_2}{G_D}+\pair{X'_3}{G_D}\right).
\end{split}
\end{equation}
This completes the proof of \eqref{eq:generalized_Polyak_formula_for_Casson}.

For 3-component links,
we can deduce \eqref{eq:generalized_Polyak_formula_for_triple_linking} by the same argument,
just replacing the segments $S_k$ with those on the $k$-th component.

\bibliographystyle{amsplain}
\bibliography{biblio}
\end{document}